\documentclass[11pt]{amsart}
\usepackage[utf8]{inputenc}

\usepackage{amsmath, amstext, amssymb, bold-extra, graphicx, tikz, amsthm, amsfonts, hyperref,longtable, bm, nicefrac,color,bbm,mathrsfs}
\usepackage[floatrow]{chemstyle}

\usepackage{mdwlist}
\usepackage{todonotes}
\usepackage{enumerate}
\RequirePackage{color}
\definecolor{my-blue}{rgb}{0.0,0.0,0.6}
\definecolor{my-red}{rgb}{0.5,0.0,0.0}
\definecolor{my-green}{rgb}{0.0,0.5,0.0}
\definecolor{nicos-red}{rgb}{0.75,0.0,0.0}
\definecolor{light-gray}{gray}{0.6}
\definecolor{really-light-gray}{gray}{0.8}
\definecolor{sussexg}{rgb}{0.0,0.5,0.5}
\definecolor{sussexp}{rgb}{0.5,0.0,0.5}

\usepackage{tikz}
\usepackage{subcaption}
\usepackage{mdwlist}
\usepackage{todonotes}
\usepackage{enumerate}
\usetikzlibrary{datavisualization} 
\usetikzlibrary{datavisualization.formats.functions} 
\usetikzlibrary{shapes.misc}
\tikzset{cross/.style={cross out, draw=black, minimum size=2*(#1-\pgflinewidth), inner sep=0pt, outer sep=0pt},
cross/.default={2pt}}

\newtheorem{theorem}{\textcolor{my-red}{\sc Theorem}}[section]
\newtheorem{lemma}[theorem]{\textcolor{my-red}{\sc { Lemma}}}
\newtheorem{proposition}[theorem]{\textcolor{my-red}{{\sc Proposition}}}
\newtheorem{corollary}[theorem]{\textcolor{my-red}{{\sc Corollary}}}

\newtheorem{definition}[theorem]{\textcolor{my-red}{{\it Definition}}}
\numberwithin{equation}{section}
\theoremstyle{remark}
\newtheorem{remark}[theorem]{Remark}

\newcommand{\be}{\begin{equation}}
\newcommand{\ee}{\end{equation}}

\providecommand{\P}[1]{\langle#1\rangle}
\newcommand{\fl}[1]{\left\lfloor{#1}\right\rfloor}

\def\bE{\mathbb{E}}
\def\bN{\mathbb{N}}
\def\bP{\mathbb{P}}

\def\bR{\mathbb{R}}
\def\bZ{\mathbb{Z}}

\def\cW{\mathcal{W}}
\def\cN{\mathcal{N}}
\def\cE{\mathcal{E}}
\def\cS{\mathcal{S}}
\def\cB{\mathcal{B}}
\def\e{\varepsilon}

\def\om{\omega}

 \def\Z{\bZ}

\def\R{\bR}
\def\N{\bN}
\def\E{\bE}
\def\P{\bP} 
\def\sS{\mathscr{S}}
\newcommand{\dis}{\overset{\text{\textit{$\mathcal{D}$}}}{=}}

\DeclareMathOperator{\Var}{Var}   
\DeclareMathOperator{\Cov}{Cov}   

\calclayout

\allowdisplaybreaks

\begin{document}
\date{\today}
\title[Variance for discrete Hammersley] 
{Order of the variance in the discrete Hammersley process with boundaries}

\author[F.~Ciech]{Federico Ciech}
\address{Federico Ciech\\ University of Sussex\\ Department of  Mathematics \\ Falmer Campus\\ Brighton BN1 9QH\\ UK.}
\email{F.Ciech@sussex.ac.uk}
\urladdr{http://www.sussex.ac.uk/profiles/395447} 

\author[N.~Georgiou]{Nicos Georgiou}
\address{Nicos Georgiou\\ University of Sussex\\ Department of  Mathematics \\ Falmer Campus\\ Brighton BN1 9QH\\ UK.}
\email{N.Georgiou@sussex.ac.uk}
\urladdr{http://www.sussex.ac.uk/profiles/329373} 

\thanks{N. Georgiou was partially supported by the EPSRC First Grant EP/P021409/1: The flat edge in last passage percolation.}

\begin{abstract}
	We discuss the order of the variance on a lattice analogue of the Hammersley process with 
	boundaries, for which the environment on each site has independent, Bernoulli distributed values. 		The last passage time is the maximum number of Bernoulli points that can be collected on a 
	piecewise linear path, where each segment has strictly positive but finite slope.
	
	We show that along characteristic directions the order of the variance of the last passage time is of
	order $N^{2/3}$ in the model with boundary. These characteristic directions are restricted in 
	a cone starting at the origin, and along any direction outside the cone, 
	the order of the variance changes to $O(N)$ in the boundary model and to $O(1)$ 
	for the non-boundary model. This behaviour is the result of the two flat edges of the shape function. 
	
\end{abstract}

\maketitle

\section{Introduction} 

\subsection{Brief description of the model and framework} This paper studies fluctuations of a corner growth model that can be viewed as a discrete analogue of the Hammersley process \cite{hammersley1972few} or an independent analogue of the longest common subsequence (LCS) problem, introduced in \cite{chvatal1975longest}. 

The model under consideration was introduced in \cite{seppalainen1997increasing}. It is a directed corner growth model on the positive quadrant $\Z_+^2$. Each site $v$ of $\Z^2_+$ is assigned a random weight $\om_v$. The collection 
$\{\om_v\}_{v \in \Z_+^2}$ is the random environment and it is i.i.d.\ under the environment measure $\P$,  with Bernoulli marginals 
\[
\P\{ \om_v = 1\} = p,  \quad \P\{ \om_v= 0\} = 1 -p.
\] 
Throughout the article we exclude the values $p = 0$ or $ p =1$. One way to view the environment, is to treat site $v$ as \emph{present} when $\om_v = 1$ and as \emph{deleted} when $\om_v = 0$. With this interpretation, the longest strictly increasing Bernoulli path up to $(m,n)$ is a sequence of present sites 
\[
L_{m,n} = \{ v_1 = (i_1, j_1),  v_2 = (i_2, j_2),  \ldots,  v_M =(i_M, j_M)\}
\] 
so that $0 < i_1 < i_2 < \ldots < i_M \le m$ and $ 0 < j_1 < j_2 < \ldots < j_M \le n$ and so that if 
$\{w_1, w_2, \ldots, w_K\}$ is a different strictly increasing sequence of present sites, then it must be 
the case that $K \le M$. 

In this paper we cast the random variable $L_{m,n}$ as a last passage time, in the framework of \cite{georgiou2016variational}. With the previous description, a step of a potential optimal path up to $(m,n)$ can take one of $O(m,n)$ values - any site is accessible as long as it has strictly larger coordinates from the previous site. However, any integer vector of positive coordinates can be written as a linear combination of $e_1, e_2$ and $e_1+e_2$ steps. Our set of admissible steps is then restricted to $\mathcal{R} = \{ e_1, e_2, e_1+e_2\}$ and an admissible path from $(0,0)$ to $(m,n)$ is an ordered sequence of sites 
\[
\pi_{0, (m,n)} = \{0 = v_0, v_1, v_2, \ldots, v_M = (m,n)\},
\]
so that $v_{k+1} - v_k \in \mathcal R$. The collection of all these paths is denoted by $\Pi_{0, (m,n)}$.
In order to obtain the same variable $L_{m,n}$ over this set of paths as the one from only strictly increasing steps, we need to specify a measurable potential function $V(\om, z):\R^{\Z^2_+}\times\mathcal R\to\R$ defined
\be\label{eq:potential}
V(\om, z) = \om_{e_1 + e_2}1\!\!1\{ z = e_1+e_2\}.
\ee
This way, the path $\pi$ will collect the Bernoulli weight at site $v$ if and only there exists a $k$ such that 
$v_{k+1} = v$ and  $v_{k} = v - e_1 - e_2$. No gain can be made through a horizontal or vertical step. Using this potential function $V$ we define the last passage time as 
\be\label{eq:lptV}
 G^V_{0,(m,n)} = \max_{\pi_{0,(m,n)} \in \Pi_{0,(m,n)}} \bigg\{ \sum_{ v_i \in \pi} V(T_{v_i}\om, v_{i+1}-v_i) \bigg\}.
 \ee
Above we used $T_{v_i}$ as the environment shift by $v_i$ in $\Z^2_+$. Now one can see that $G^V_{0,(m,n)} = L_{m,n}$.   

The law of large numbers for $G^V_{0,(m,n)}$ was first obtained in \cite{seppalainen1997increasing} by first obtaining invariant distributions for an embedded totally asymmetric particle system. It is precisely this methodology that invites the characterization `discrete Hammersley process' as the particle system can be viewed as a discretized version of the Aldous-Diaconis process \cite{aldous1995hammersley} which find the law of large numbers limit for the number of Poisson(1) points that can be collected from a strictly increasing path in $\R^2_+$. 

The original problem is mentioned as Ulam's problem in the literature and it was about the limiting law of large numbers for the longest increasing subsequence of a random permutation of the first $n$ numbers, denoted by $I_n$. Already in \cite{Erd-Sze-35} it was shown that $I_n \ge \sqrt{n}$ and an elementary proof via a pigeonhole argument can be found in \cite{hammersley1972few}. This gave the correct scaling and it was proven in \cite{Log-She-77, Ver-Ker-77}  that limiting constant is 2. Then the combinatorial arguments of these papers where changed to softer probabilistic arguments in \cite{aldous1995hammersley, Sep-96, Gro-01} where the full law of large numbers was obtained for sequence of increasing Poisson points.  


For the discrete Hammersley the law of large numbers for the point-to-point shape function $g^{(p)}_{pp}(s,t)$ was computed in \cite{seppalainen1997increasing} to be
\be\label{eq:}
g^{(p)}_{pp}(s,t) = \lim_{n\to \infty} \frac{ G^V_{0,(\fl{ns},\fl{nt})}}{n}
=\begin{cases}
s, \qquad &t\geq \frac{s}{p},\\
\frac{1}{1-p}\big(2\sqrt{pst}-p(t+s)\big), \qquad &ps\leq t<\frac{s}{p},\\
t, \qquad& t\leq ps.
\end{cases}
\ee
This is a concave, symmetric, 1-homogeneous differentiable function which is continuous up to the boundaries of $\R^2_+$ and it was the first completely explicit shape function for which strict concavity is not valid. In fact, the formula indicates two flat edges, for $t > s/p$ or $t < ps$.  

The argument used in  \cite{seppalainen1997increasing} to obtain the formula in directions of the flat edge can also be used in an identical way to obtain the law of large numbers in the same direction for the much more correlated LCS model \cite{chvatal1975longest}. Comparisons between the discrete Hammersley and the LCS are tantilizing. The Bernoulli environment $\eta = \{\eta_{i,j}\}$ for the LCS model is uniquely determined by two infinite random strings 
${\bf x} = (x_1, x_2, \ldots)$ and  ${\bf y} = (y_1, y_2, \ldots)$ where each digit is uniformly chosen from a $k$-ary alphabet (i.e.\ $x_i, y_j \in\{1, 2, \ldots, k \}$). Then the environment $\eta_{i,j} = \mathbbm1\{ x_i = y_j\}$ and it takes the value $1$ with probability $p = 1/k$. 
The random variable $\mathcal L_{n,n}^{(k)}$ represents the longest increasing sequence of Bernoulli points in this environment, which corresponds to the longest common subsequence between the two words, of size $n$. The limit $c_k = \lim_{n \to \infty} n^{-1} \mathcal L_{n,n}^{(k)} $ is called in the literature as the the Chvatal-Sankoff constant, and it was already observed in \cite{seppalainen1997increasing} that  $g^{(1/k)}_{pp}(1,1)$ of the discrete Hammersley lies between the known computational upper and lower bounds for $c_k$. 
 
A formal connection between the  discrete Hammersley, LCS and Hammersley models arises in the small $p$ (large alphabet size $k$) limit.  Sankoff and Mainville conjectured in \cite{sankoff1983common} that 
\[\lim_{k\to\infty}\frac{c_k}{\sqrt{k}}=2.\]
For the discrete Hammersley model this is an immediate computation in \eqref{eq:} for $p= 1/k$ when we change $c_k$ with $g^{(1/k)}_{pp}(1,1)$. For the LCS, this was proven in \cite{kiwi2005expected}. The value 2 is the limiting law of large numbers value for the longest increasing sequence of Poisson points in $\R^2_+$. 

\subsection{Solvable models of lattice last passage percolation and KPZ exponents}

Identifying the explicit shape function is the first step in computing fluctuations and scaling limits for last passage time quantities. When precise calculations can be performed and explicit scaling laws can be computed the model is classified as an explicitly solvable model of last passage percolation. There is only a handful of these models, and each one requires individual treatment.
 
In \cite{Bai-Dei-Joh-99} it is proven that the fluctuations around the mean of the LIS of $n$ numbers are of order $n^{1/6}$ and the scaling limit is a Tracy-Widom distribution using a determinantal approach. The fluctuation exponent $1/3$ is often used to associate a model to the Kardar-Parisi-Zhang (KPZ) class (see \cite{corwin2012kardar} for a review), and determinental/combinatorial approaches were developed for a variety of solvable growth models in order to compute among other things explicit weak limits and formulas for Laplace transforms of last passage times and polymer partition functions. Lattice examples include the corner growth model with i.i.d.\ geometric weights, (admissible steps $e_1, e_2$) \cite{Joh-00}, the log-gamma polymer \cite{Bor-Cor-Rem-13, Cor-etal-14}, introduced in \cite{Sep-12}, the Brownian polymer \cite{Mor-OCo-07, Sep-Val-10}, the strict-weak lattice polymer \cite{corwin2015strict, OCo-Ort-14b}, the random walk in a beta-distributed random potential, where the zero-temperature limit is the Bernoulli-Exponential first passage percolation \cite{Bar-Cor-15-}. The result of \cite{Joh-00} was also used to derive explicit formulas for the discrete Hammersley {\cite{priezzhev2008exact} with no boundaries via a particle system coupling.


It is expected that under some minimal moment conditions the order of fluctuations of $1/3$ of the last passage time or the polymer partition function is environment-independent. A general theory that is a step towards universality can be found at the law of large numbers level \cite{georgiou2016variational, rassoul2014quenched, rassoul2013quenched, rassoul2017variational} where a series of variational formulas for the limiting free energy density of polymer models and shape functions for last passage percolation where proven. A variational formula for the time constant in first passage percolation was proven in  \cite{Kri-15-}. For two-dimensional last passage models with $e_1, e_2$ admissible steps the analysis and results can be sharpened; early universal results on the shape near the edge were obtained in \cite{Martin2004limiting, bodineau2005universality}. A general approach and a range of results including solutions to the variational formulas and existence of directional geodesics using invariant boundary models were developed via the use of cocycles in \cite{Geo-Ras-Sep-15b-} and \cite{Geo-Ras-Sep-15a-}. Similar techniques are utilized in the present article, since we prove the existence of an invariant boundary model for the discrete Hammersley.
 
A more probabilistic approach to estimate the order of the variance (but not the explicit scaling limit), was developed in \cite{Cat-Gro-06} and \cite{groeneboom2002hydrodynamical} where by adding Poisson distributed `sinks' and `sources' on the axes, they could create invariant versions of the model. For the discrete Hammersley, an invariant model with sinks and sources has been described in \cite{basdevant2015discrete} and it was used to re-derive the law of large numbers for $G^V_{0,(m,n)}$. In the present article we show another way to use boundaries on the axes and create invariant boundary models. Our approach is similar to those in \cite{balazs2006cube, Sep-12, Sep-Val-10} where a Burke type property is first proven for the model with boundary and then exploited to obtain the order of fluctuations. 

\subsection{The flat edge in lattice percolation models} 

The discrete Hammersley is a model for which the shape function $g_{pp}(s,t)$
exhibits two flat edges, for any value of $p$. Flat edge in percolation is not uncommon. A flat edge for the contact process was observed in \cite{durrett1981shape} and \cite{durrett1983supercritical}. A simple explicitly solvable first passage (oriented) bond percolation model introduced in \cite{Sep-98-aop-2} allows for an exact derivation of the limiting shape function and it also exhibits a flat edge. In this model the random weight was collected only via a horizontal step, while vertical steps had a deterministic cost. For the i.i.d.\ oriented bond percolation where each lattice edge admits a random Bernoulli weight, a flat edge result for the shape was proved in \cite{durrett1984oriented} when the probability of success $p$ is larger than some critical value and percolation occurs. This was later extended in \cite{Mar-02} where further properties were derived. In \cite{Auf-Dam-13} differentiability at the shape at the edge of flat edge was proven. 

These properties for oriented bond percolation can be transported to oriented site percolation and further extended to corner growth models when the environment distribution has a percolating maximum. For a general treatment to this effect, for non-exactly solvable models, see Section 3.2 in \cite{Geo-Ras-Sep-15a-}. For directed percolation in a correlated environment, a shape result with flat edges can be found in \cite{Emr-16}. 

Local laws of large numbers of the passage time near the flat edge of the discrete Hammersley model can be found in \cite{georgiou2010soft}. This work was later extended in \cite{georgiou2016optimality}, where limiting Tracy-Widom laws were obtained in special cases, using also the edge results of \cite{bodineau2005universality}. These `edge results' are for the last passage time in directions that are below the critical line $(n, n/p)$ and into the concave region of $g_{pp}$ by a mesoscopic term of $n^a$, $0< a < 1$. When $a > 1/2$ the order of the fluctuations is between $O(n^{1/3})$ and $O(1)$. In the present article we further prove that in directions above the critical line (in the flat edge of $g_{pp}$) the variance of the passage time is bounded above by a constant that tends to 0 (see Section \ref{sec:noboundary}).

\subsection{Structure of the paper}
The paper is organised as follows:
 In Section \ref{sec:model} we state our main results after describing the boundary model. 
 In Section \ref{sec:burke} we prove Burke's property for the invariant boundary model 
  and compute the solution to the variational formula that gives the law of large numbers for the shape function of the model without boundaries. 
  The main theorem of this paper is the order of the variance of the model with boundaries in characteristic directions. 
  The upper bound for the order can be found in Section \ref{sec:varub}. The lower bound is proven in Section \ref{sec:varlb}.
  For the order of the variance in off-characteristic directions see Section \ref{sec:offchar} and for the results for the model with no boundaries, including the order of the variance in directions in the flat edge see Section \ref{sec:noboundary}.
 Finally, in Section \ref{sec:paths} we prove the path fluctuations in the characteristic direction, again in the model with boundaries.

\subsection{Common notation}
Throughout the paper, $\N$ denotes the natural numbers, and $\Z_+$ the non-negative integers. When we write inequality between two vectors $v = (k, \ell) \le  w = (m,n)$ we mean $k \le m$ and $\ell \le n$. We reserve the symbol $G$ for last passage times. We omit from the notation the superscript $V$ that was used to denote the dependence of potential function in \eqref{eq:lptV}, since for the sequence we fix $V$ as in \eqref{eq:potential}, unless otherwise mentioned. The symbol $\pi$ is reserved for a generic admissible path.

\section{The model and its invariant version} 
\label{sec:model}
%
 
 \subsection{The invariant boundary model} The boundary model has altered distributions of weights on the two axes. The new environment there will depend on a parameter $u \in (0,1)$ that will be under our control. Each $u$ defines different boundary distributions. At the origin we set  $\om_0 = 0$.
For weights on the horizontal axis, for any $k \in \N$ we set
$
\om_{ke_1} \sim \text{Bernoulli}( u ), 
$
with independent marginals 
\be \label{eq:xaxis}
\P\{ \om_{ke_1} = 1\} = u= 1 - \P\{ \om_{ke_1} = 0\}.
\ee
On the vertical axis, for any $k \in \N$, we set 
$\om_{ke_2} \sim \text{Bernoulli}\Big( \frac{p(1-u)}{u + p(1-u)} \Big)$ with independent marginals 
\be\label{eq:yaxis} 
\P\{ \om_{ke_2} = 1 \} 
=  \frac{p(1-u)}{u + p(1-u)} =  1 - \P\{ \om_{ke_2} = 0\}.
\ee
The environment in the bulk $\{ \om_w\}_{w \in \N^2}$ remains unchanged with i.i.d.\ Ber($p$) marginal distributions. 
Denote this environment by $\om^{(u)}$ to emphasise the different distributions on the axes that depend on $u$. 

In summary, for any $i \ge 1, j \ge 1$, the $\om^{(u)}$ marginals are independent under a background environment measure $\P$ with marginals
\be \label{eq:omu}
\om^{(u)}_{i,j} \sim
\begin{cases}
\text{Ber}(p), &\text{ if } (i,j) \in \N^2,\\
\text{Ber}(u), &\text{ if }  i \in \N, j =0,\\
\text{Ber}\Big(\frac{p(1-u)}{u + p(1-u)} \Big), &\text{ if }  i =0, j \in \N, \\
0, &\text{ if }  i =0, j =0.
\end{cases}
\ee

In this environment we slightly alter the way a path can collect weight on the boundaries. Consider any path $\pi$ from $0$. If the path moves horizontally before entering the bulk,  then it collects the Bernoulli($u$) weights until it takes the first vertical step, and after that, it collects weight according to the potential function  \eqref{eq:potential}. If $\pi$ moves vertically from $0$ then it \textbf{also} collects the Bernoulli weights on the vertical axis, and after it enters the bulk, it collects according to \eqref{eq:potential}. 

Fix a parameter $u \in (0, 1)$. Denote the last passage time from 0 to $w$ in environment $\om^{(u)}$ by $G^{(u)}_{0, w}$. The variational equality, using the above description, is 
\begin{align}
G^{(u)}_{0, w} &= \max_{1\le k \le w\cdot e_1} \Big\{ \sum_{i=1}^k \om_{i e_1} + G_{ke_1+e_2, w} \Big\} \nonumber\\
&\phantom{xxxxxxxxxxxxx}\bigvee \max_{1\le k \le w\cdot e_2} \Big\{ \sum_{j=1}^k \om_{je_2}  + \om_{e_1+ke_2}+G_{e_1+ke_2, w} \Big\}. \label{eq:varform}
\end{align}

Our two first statements give the explicit formula for the shape function. 

  \begin{theorem}\label{thm:LLNG}[Law of large numbers for $G^{(u)}_{\fl{Ns}, \fl{Nt}}$]
	For fixed parameter $0 < u \le 1$ and $(s,t) \in \R^2_+$ we have 
	\be
	\lim_{N\to \infty} \frac{G^{(u)}_{\fl{Ns}, \fl{Nt}}}{N} = su + t\, \frac{p(1-u)}{ u + p(1-u) } , \quad  \P -a.s.
	\ee 
\end{theorem}

\begin{theorem}[\cite{seppalainen1997increasing}, \cite{basdevant2015discrete}]
\label{thm:LLNp}
Fix $p$ in $(0,1)$ and $(s,t) \in \R^2_+$. Then we have the explicit law of large numbers limit
\begin{align}
\lim_{N \to \infty} \frac{G_{\fl{Ns}, \fl{Nt}}}{N} &= \inf_{0< u \le 1}\{ s \E(\om^{(u)}_{1,0}) + t \E(\om^{(u)}_{0,1})\} \notag\\
&=\begin{cases}
s, \qquad &t\geq \frac{s}{p}\\
\frac{1}{1-p}\big(2\sqrt{pst}-p(t+s)\big), \qquad &ps\leq t<\frac{s}{p}\\
t, \qquad& t\leq ps.
\end{cases}\label{eq:busopt}
\end{align}
\end{theorem}

The main theorems of this article verify with probabilistic techniques the variance of $G^{(u)}$ along deterministic directions. For a given boundary parameter $u$, there will exist a unique direction $(m_u, n_u)$ along which the last passage time at point $N(m_u, n_u)$ time will have variance of order $O(N^{2/3})$ for large $N$. That is what we call the characteristic direction. The form of the characteristic direction will become apparent from the variance formula in Proposition \ref{lem:exit}; it is precisely the direction for which the higher order variance terms cancel out. As it turns out, the characteristic direction ends up being 
\begin{equation}\label{eq:bound}
(m_{u}(N), n_{u}(N))=\left(N, \Big\lfloor \frac{N}{p}\big(p + (1-p)u\big)^2\Big\rfloor\right).
\end{equation}
Throughout the paper we will often compare last passage times over two different boundaries that have different characteristic directions. For this reason we explicitly denote the parameter in the subscript.

Note that as $N \to \infty$, the scaled direction converges to the macroscopic characteristic direction
\be\label{eq:chardirmacro}
N^{-1}(m_{u}(N), n_{u}(N)) \to \Big( 1, \frac{\big(p+(1-p)u\big)^2}{p}\Big),
\ee
which gives that for large enough $N$ the endpoint $(m_{u}(N), n_{u}(N))$ is always between the two critical lines $y = \frac{x}{p}$ and $y=px$ that separate the flat edges from the strictly concave part of $g_{pp}$.This defines the macroscopic set of characteristic directions
\[
\frak J_p = \Big\{ \Big(1,  \frac{\big(p+(1-p)u\big)^2}{p}\Big): u \in (0,1) \Big\}. 
\]
Note that any $(s,t) \in\R^2_+$ for which $(1,ts^{-1} ) \in \frak J_p$, the shape function $g_{pp}$ has a strictly positive curvature at $(s,t)$.

\begin{theorem}\label{thm:varub} Fix a parameter $u \in (0,1)$ and let $(m_u, n_u)$ the characteristic direction corresponding to $u$ as in \eqref{eq:chardirmacro} and large scale approximation,  $(m_u(N), n_u(N))$ as in \eqref{eq:bound}. Then there exists constants $C_1$ and $C_2$ that depend on $p$ and $u$ so that 
\be
C_1 N^{2/3} \le \Var( G^{(u)}_{m_u(N), n_u(N)}) \le C_2N^{2/3}.
\ee 
\end{theorem}

In the off-characteristic direction, the process $G^{(u)}_{m_u(N), n_u(N)}$  satisfies a central limit theorem, and therefore the variance is of order $N$. This is due to the boundary effect, as we show that maximal paths spend a macroscopic amount of steps along a boundary, and enter the bulk at a point which creates a characteristic rectangle  with the projected exit point. 

\begin{theorem}\label{thrm:voffchar}
Fix a $c\in \R$. Fix a parameter $u \in (0,1)$ and let $(m_u, n_u)$ the characteristic direction corresponding to $u$ as in \eqref{eq:bound}. Then for $\alpha \in (2/3, 1]$,
\begin{align*}
&\lim_{N\to\infty}\frac{G^{(u)}_{m_u(N), n_u(N)+\fl{c N^\alpha}}-\E[G^{(u)}_{m_u(N),  n_u(N)+\fl{c N^\alpha}}]}{N^{\alpha/2}}\\
&\hspace{5cm}\stackrel{\mathcal D}{\longrightarrow} Z \sim \mathcal N(0,\Var(\om^{(u)}_{1,0}) \mathbbm{1}\{c<0\}+\Var(\om^{(u)}_{0,1}) \mathbbm{1}\{c>0\}).
\end{align*}
\end{theorem}

\begin{remark} Set $\frak J_p$ contains only the directions $(1, t)$ for which 
$p < t < 1/p$. Any other directions with $t < p$ or $t > p^{-1}$ -that also correspond to the flat edge of the non-boundary model- and for an arbitrary $u \in (0,1)$, are necessarily off-characteristic directions and along those, the last passage time satisfies a central limit theorem.  \qed
\end{remark}

We also have partial results for the model without boundaries. The approach does not allow to access the variance of the non-boundary model directly, but we have 

\begin{theorem}\label{thm:nbsl}
Fix $x,y\in(0,\infty)$ so that $p < y/x < p^{-1}$. Then, there exist finite constants $N_0$ and $C=C(x, y, p)$, such that, for $b\geq C$,  $N\geq N_0$ and any $0<\alpha<1$,
\begin{equation}\label{eq:ppp}
\P\{|G_{(1,1),(\lfloor Nx\rfloor,\lfloor Ny\rfloor)}-Ng_{pp}(x,y)|\geq bN^{1/3}\}\leq Cb^{-3\alpha/2}.
\end{equation}

In particular, for $N > N_0$, and $1\leq r<3\alpha/2$  we get the moment bound 
\begin{equation}
\E\bigg[\bigg|\frac{G_{(1,1),(\lfloor Nx\rfloor,\lfloor Ny\rfloor)}-Ng_{pp}(x,y)}{N^{1/3}}\bigg|^r\bigg]\leq C(x,y,p,r)<\infty.
\end{equation}
\end{theorem}

The bounds in the previous theorem work in directions where the shape function is strictly concave. In directions of flat edge we have

\begin{theorem}\label{thm:flatvar}
Fix $x,y\in(0,\infty)$ so that $p > y/x$ or $ y/x > p^{-1}$. Then, there exist finite constants $c = c(x,y,p)$ and $C=C(x, y, p)$, such that
\begin{equation}\label{eq:ppp2}
\Var(G_{(1,1),(\lfloor Nx\rfloor,\lfloor Ny\rfloor)}) \le C N^2 e^{-cN}\to 0 \quad (N \to \infty).
\end{equation}
\end{theorem}
For finer asymptotics on the variance and also weak limits, particularly close to the critical lines $y = px$ and $y = p^{-1}x$ we direct the reader to \cite{georgiou2010soft, georgiou2016optimality}.
 
We have already alluded to the maximal paths. Maximal paths are admissible paths that attain the last passage time. In the literature they can also be found as random geodesics. For this particular model, the maximal path is not unique - this is because of the discrete nature of the environment distribution, so we need to enforce an a priori
condition that makes our choice unique when we refer to it.  Unless otherwise specified, the maximal path we select is the \emph{right-most} one (it is also the \emph{down-most} maximal path). 

\begin{definition} An admissible maximal path from $0$ to $(m,n)$
\[\hat \pi_{0, (m,n)} = \{ \{ (0,0) = \hat \pi_0, \hat \pi_1, \ldots, \hat \pi_K = (m,n)  \}\]
 is the right-most (or down-most) maximal path if and only if it is maximal and if $ \hat \pi _i = (v_i,w_i) \in \hat \pi_{0, (m,n)}$ then the sites $( k,\ell), v_i < k < m, 0 \le \ell < w_i$  cannot belong on any maximal path from 0 to $(m,n)$.
\end{definition}
In words, no site underneath the right-most maximal path can belong on a different maximal path. An algorithm to construct the right-most path iteratively is given in \eqref{maxpathhatpi}.

For this right-most path $\hat \pi$ we define $\xi^{(u)}$ its exit point from the axes in the environment $\om^{(u)}$. We indicate with $\xi^{(u)}_{e_1}$ the exit point from the $x$-axis and $\xi^{(u)}_{e_2}$ the exit point from the $y$-axis. If $\xi^{(u)}_{e_1}>0$ the maximal path $\hat\pi$ chooses to go through the $x-$axis and $\xi^{(u)}_{e_2}=0$ and vice versa. If $\xi^{(u)}_{e_1}=\xi^{(u)}_{e_2}=0$ it means the maximal path directly enters into the bulk with a diagonal step. When we do not need to distinguish from which axes we exit, we just denote the generic exit point by $\xi^{(u)}$. 

The exit point $\xi^{(u)}_{e_1}$ represents the exit of the maximal path from level 0. To study the fluctuations of this path around its enforced direction,  
define 
\begin{equation}\label{v1}
v_0(j)=\min\{i\in\{0,\dots,m\}:\exists k\text{ such that } \hat \pi_k=(i,j)\},
\end{equation}
and 
\begin{equation} \label{v2}
v_1(j)=\max\{i\in\{0,\dots,m\}:\exists k\text{ such that } \hat \pi_k=(i,j)\}.
\end{equation}
These represent, respectively, the entry and exit point from a fixed horizontal level $j$ of a path $\hat \pi$. Since our paths can take diagonal steps, it may be that $v_0(j) = v_1(j)$ for some $j$.  

Now, we can state the theorem which shows that $N^{2/3}$ is the correct order of the magnitude of the path fluctuations. We show that the path stays in an $\ell^1$  ball of radius $CN^{2/3}$ with high probability, and simultaneously, avoid balls of radius $\delta N^{2/3}$ again with high probability for $\delta$ small enough.  

\begin{theorem} \label{thm:pathflucts}
Consider the last passage time in environment $\om^{(u)}$ and let $\hat\pi_{0,m_u(N),n_u(N)}$ be the right-most maximal path from the origin up to  $(m_u(N),n_u(N))$ as in \eqref{eq:bound}.    
Fix a  $0\leq\tau<1$. Then, there exist constants $C_1,C_2<\infty$ such that for $N\geq 1$, $b\geq C_1$ 
\begin{equation}\label{eq:path1}
\P\{v_0(\lfloor\tau n_u(N)\rfloor)<\tau m_u(N)-bN^{2/3}\text{ or }v_1(\lfloor\tau n_u(N)\rfloor)>\tau m_u(N)+bN^{2/3}\}\leq C_2b^{-3}.
\end{equation}
The same bound holds for vertical displacements.\\
Moreover, for a fixed $\tau \in (0,1)$ and given $\e>0$, there exists $\delta>0$ such that
\begin{equation}\label{eq:path2}
\lim_{N\to\infty}\P\{\exists k \text{ such that } |\hat \pi_k-(\tau m_u(N),\tau n_u(N))|\leq \delta N^{2/3}\}\leq \e.
\end{equation}
\end{theorem}
\bigskip

\section{Burke's property and law of large numbers}
\label{sec:burke}
To simplify the notation in what follows, set $w = (i,j) \in \Z_+^2$ and define the last passage time gradients by 
\be\label{eq:gradients}
I^{(u)}_{i+1,j} = G^{(u)}_{0, (i+1,j)}  - G^{(u)}_{0, (i,j)}, \quad \text{and} \quad J^{(u)}_{i,j+1} =  G^{(u)}_{0, (i,j+1)}  - G^{(u)}_{0, (i,j)}.
\ee
When there is no confusion we will drop the superscript $(u)$ from the above. When $j = 0$ we have that $\{I^{(u)}_{i,0}\}_{i, \in \N}$ is a collection of i.i.d.\ Bernoulli($u$) random variables since $I^{(u)}_{i,0} = \om_{(i,0)}$. Similarly, for $i = 0$,   $\{J^{(u)}_{0,j}\}_{j \in \N}$ is a collection of i.i.d.\ $\text{Bernoulli}\Big(\frac{p(1-u)}{u + p(1-u)} \Big)$ random variables.

The gradients and the passage time satisfy recursive equations. This is the content of the next lemma. 
\begin{lemma}
Let $u \in (0,1)$  and $(i, j) \in \N^2$. Then the last passage time can be recursively computed as 
\begin{align}\label{eq:LPPrec}
G^{(u)}_{0, (i, j)} &= \max\big\{ G^{(u)}_{0, (i, j-1)}, \,\,G^{(u)}_{0, (i-1, j)},  \,\,G^{(u)}_{0, (i-1, j-1)} + \om_{i,j} \big\} 
\end{align}
Furthermore, the last passage time gradients satisfy the recursive equations 
\be\label{eq:4}
\begin{aligned}
I^{(u)}_{i,j}&=\max\{\omega_{i,j}, \, J^{(u)}_{i-1,j}, \, I^{(u)}_{i,j-1}\} - J^{(u)}_{i-1,j} \\
J^{(u)}_{i,j}&= \max\{\omega_{i,j}, \, J^{(u)}_{i-1,j}, \, I^{(u)}_{i,j-1}\} - I^{(u)}_{i,j-1}. 
\end{aligned}
\ee
\end{lemma}

\begin{proof}
Equation \eqref{eq:LPPrec} is immediate from the description of the dynamics in the boundary model and the fact that $(i,j)$ is in the bulk. 
We only prove the recursive equation \eqref{eq:4} for the $J$ and the other one is done similarly and left to the reader. Compute 
\begin{align*}
J^{(u)}_{i,j}&=G^{(u)}_{0,(i,j)}-G^{(u)}_{0, (i,j-1)}\\
&=\max\big\{ G^{(u)}_{0, (i, j-1)}, \,\,G^{(u)}_{0, (i-1, j)},  \,\,G^{(u)}_{0, (i-1, j-1)} + \om_{i,j} \big\}  - G^{(u)}_{0,(i,j-1)} \quad \text{ by \eqref{eq:LPPrec},}\\
&=\max\big\{0, G^{(u)}_{0, (i-1,j)}-G^{(u)}_{0, (i,j-1)},  G^{(u)}_{0,(i-1,j-1)}-G^{(u)}_{0, (i,j-1)}+\om_{i,j}\big\}\\
&=\max\big\{0, G^{(u)}_{0, (i-1,j)} \pm G^{(u)}_{0,(i-1,j-1)} -G^{(u)}_{0, (i,j-1)},  G^{(u)}_{0,(i-1,j-1)}-G^{(u)}_{0, (i,j-1)}+\om_{i,j}\big\}\\
&=\max\big\{0, J^{(u)}_{i-1,j} -   I^{(u)}_{i,j-1}, - I^{(u)}_{i,j-1} +\om_{i,j}\big\}\\
&= \max\{\omega_{i,j}, \, J^{(u)}_{i-1,j}, \, I^{(u)}_{i,j-1}\} - I^{(u)}_{i,j-1}. 
 \qedhere
\end{align*}
\end{proof}

The recursive equations are sufficient to prove a partial independence property. 
\begin{lemma}\label{lem:parbur} Assume that $(\omega_{i,j}, I^{(u)}_{i, j-1}, J^{(u)}_{i-1,j})$ are mutually independent with marginal  distributions given by 
\begin{equation}\label{eq:bnddist}
\omega_{i,j}\sim \text{Ber}(p), \quad I^{(u)}_{i, j-1} \sim \text{Ber}(u), \quad J^{(u)}_{i-1,j} \sim  \text{Ber}\Big(\frac{p(1-u)}{u + p(1-u)}\Big).
\end{equation}
Then, $I^{(u)}_{i,j}, J^{(u)}_{i,j}$, computed using the recursive equations \eqref{eq:4} are independent with marginals Ber$(u)$ and Ber$(\frac{p(1-u)}{u + p(1-u)})$ respectively. 
 \end{lemma}
 
\begin{proof}
The marginal distributions are immediate from the definitions and the independence follows when one shows
\allowdisplaybreaks
\begin{align*}
\mathbb{E}&(h(I^{(u)}_{i,j})k(J^{(u)}_{i,j})) \\
&= \mathbb{E}(h( \omega_{i,j}\vee J^{(u)}_{i-1,j} \vee I^{(u)}_{i,j-1} - J^{(u)}_{i-1,j}) k (\omega_{i,j}\vee J^{(u)}_{i-1,j} \vee I^{(u)}_{i,j-1} - I^{(u)}_{i,j-1}))\\
&= \mathbb{E}(h(I^{(u)}_{i,j-1})k(J^{(u)}_{i-1,j})).
\end{align*}
for any bounded continuous functions $h$, $k$. We omit the details, as they are similar to the proof of Lemma \ref{burke} below. However, in order to prove  Lemma \ref{burke}, one first needs to prove Lemma \ref{lem:parbur}.
\end{proof}

A down-right path $\psi$ on the lattice $\Z^2_+$ is an ordered  sequence of sites $\{ v_i \}_{i \in \Z}$ that satisfy 
\be\label{eq:psi}
 v_i - v_{i-1} \in \{ e_1, - e_2 \}.
\ee
For a given down-right path $\psi$, define $\psi_i = v_i - v_{i-1}$ to be the $i$-th edge of the path and set  
\be \label{eq:patheses}
L_{\psi_i} = 
\begin{cases}
	I^{(u)}_{v_i}, & \text{if } \psi_i = e_1\\
	J^{(u)}_{v_{i-1}}, & \text{if } \psi_i = -e_2. 
\end{cases}
\ee
The first observation is that the random variables in the collection $\{ L_{\psi_i}\}_{i \in \Z}$ satisfy the following:
\begin{lemma}\label{lem:drp-bus}
Fix a down-right path $\psi$. Then the random variables 
	$ \{L_{\psi_i}\}_{i \in \Z}$
	are mutually independent, with marginals 
	\[
	  L_{\psi_i} \sim 
\begin{cases}
	\text{Ber}(u), & \text{if } \psi_i = e_1\\
	\text{Ber}\Big( \frac{p(1-u)}{ u + p(1-u)} \Big), & \text{if } \psi_i = -e_2. 
\end{cases}
	\]
\end{lemma}
 
 \begin{proof}
 The proof goes by an inductive ``corner - flipping" argument:  The base case is the path that follows the axes, and there the result follows immediately by the definitions of boundaries. Then we flip the corner at zero, i.e. we consider the down right path 
 \[
 \psi^{(1)} = \{ \ldots, (0,2), (0,1), (1,1), (1,0), (2,0), \ldots \}.
 \]
 Equivalently, we now consider the collection
 $\big\{ \{ J^{(u)}_{0,j} \}_{j \ge 2},\, I^{(u)}_{1,1},  J^{(u)}_{1,1},  \{ I^{(u)}_{i, 0} \}_{i \ge 2} \big\}$.
 The only place where the independence or the distributions may have been violated, is for $I^{(u)}_{1,1}$,  $J^{(u)}_{1,1}$. Lemma \ref{lem:parbur} shows this does not happen. 
 As a consequence, variables on the new path satisfy the assumption of Lemma \ref{lem:parbur}. 
 We can now repetitively use Lemma \ref{lem:drp-bus} by flipping down-right west-south corners into north-east corners. This way, starting from the axes we can obtain any down-right path, while the distributional properties are maintained. The details are left to the reader. 
\end{proof}

For any triplet $(\omega_{i,j}, I^{(u)}_{i-1,j}, J^{(u)}_{i,j-1})$ with $i \ge 1, j\ge 1$, we define the event 
\be \label{eq:theevent}
\mathcal B_{i,j} = \big\{  (\omega_{i,j}, I^{(u)}_{i-1,j}, J^{(u)}_{i,j-1}) \in (1,0,0), (0,1,0), ( 0,0,1), (1,0,1), (1,1,0) \}.
\ee
Using the gradients \eqref{eq:4}, the environment $\{\om_{i,j}\}_{(i,j) \in \N^2}$ and the events $\mathcal B_{i,j}$ we also define new random variables $\alpha_{i,j}$ on $\Z_+^2$
\be
\label{eq:5}
\alpha_{i-1,j-1}= 1\!\!1\{ I^{(u)}_{i,j-1}=J^{(u)}_{i-1,j} =1\} + \beta_{i-1,j-1} 1\!\!1\{ \mathcal B_{i,j}\} \qquad \text{for }(i,j)\in\mathbb{N}^2.
\ee
$\beta_{i-1,j-1}$ is a $\text{Ber}(p)$ random variable and is independent of everything else. Note that $\alpha_{i-1,j-1}$ is automatically $0$ when $\omega_{i,j} =  I^{(u)}_{i,j-1}=J^{(u)}_{i-1,j} = 0$ and check, with the help of Lemma \ref{lem:parbur}, that  $\alpha_{i-1,j-1} \stackrel{D}{=} \omega_{i,j}$. 
The following lemma gives the distribution of the triple 
$(I^{(u)}_{i,j}, J^{(u)}_{i,j}, \alpha_{i-1,j-1})$. It is an analogue of Burke's property for $M/M/1$ queues. 

\begin{lemma}[Burke's property] \label{burke}
Let $(\omega_{i,j}, I^{(u)}_{i,j-1}, J^{(u)}_{i-1,j})$ mutually independent Bernoulli random variables with distributions 
\[
\omega_{i,j} \sim \text{Ber}(p), \quad I^{(u)}_{i,j-1} \sim \text{Ber}(u), \quad J^{(u)}_{i-1,j} \sim \text{Ber}\big( \frac{p(1-u)}{u + p(1-u)}\big).
\]
Then the random variables $(\alpha_{i-1,j-1}, I^{(u)}_{i,j}, J^{(u)}_{i,j})$ are mutually independent with marginal distributions 
 \[
\alpha_{i-1,j-1} \sim \text{Ber}(p), \quad I^{(u)}_{i,j} \sim \text{Ber}(u), \quad J^{(u)}_{i,j} \sim \text{Ber}\big( \frac{p(1-u)}{u + p(1-u)}\big).
\]
\end{lemma}

\begin{proof}
Let $g, h, k$ be bounded continuous functions. To simplify the notation slightly, set  $\ell = \ell(u) =  \frac{p(1-u)}{u + p(1-u)}$. In the computation below we use equations \eqref{eq:4} without special mention. 
\begin{align*} 
\mathbb E &(g(\alpha_{i-1, j-1})h(I^{(u)}_{i, j})k(J^{(u)}_{i,j})) \\
&= g(1)\mathbb E\Big(h(I^{(u)}_{i, j})k(J^{(u)}_{i,j})1\!\!1\{\ I^{(u)}_{i,j-1} = J^{(u)}_{i-1,j} =1 \}) \\
&\phantom{xxxxxxx} +g(0) \mathbb E\Big(h(I^{(u)}_{i, j})k(J^{(u)}_{i,j})1\!\!1\{\omega_{i,j} = I^{(u)}_{i,j-1} = J^{(u)}_{i-1,j} = 0\} \Big)\\
&\phantom{xxxxxxxxxxxxxxxxxxxxxxxxx} + \mathbb E\Big(g(\beta_{i,j})h(I^{(u)}_{i, j})k(J^{(u)}_{i,j}) 1\!\!1\{ \cB_{i,j}\}\Big)\\
&= g(1) h(0) k(0) u\ell +g(0) h(0)k(0)(1-p)(1-u)(1-\ell)\\
&\phantom{xxxxxxxxxxxxxx} +\mathbb E(g(\beta_{i,j})\mathbb E\Big(h(I^{(u)}_{i, j})k(J^{(u)}_{i,j}) 1\!\!1\{ \cB_{i,j}\}\Big)\\
&= h(0)k(0)(1-u)(1-\ell) ( pg(1) + (1-p) g(0))\\
&\phantom{xxxxxxxxxxxxxxxxxxxxxxxxx} + \mathbb E(g(\beta_{i,j})) \sum_{x \in \cB_{i,j}}\mathbb E \Big(h(I^{(u)}_{i,j})k(J^{(u)}_{i,j})1\!\!1\{ x \in \cB_{i,j}\}\Big
)\\
&= h(0)k(0)(1-u)(1-\ell) ( pg(1) + (1-p) g(0))\\
&\phantom{xxx} + \mathbb E(g(\beta_{i,j})) \\
&\phantom{xxxxxx}\times \Big( h(1)k(1) p (1-u)(1-\ell) + h(0)k(1)[(1-p)(1-u)\ell + p(1-u)\ell)] \\
&\phantom{xxxxxxxxxxxxxxxxxxxxxxxxxxxxxx}+h(1)k(0)[(1-p)u(1-\ell) + pu(1-\ell)] \Big)\\
&= h(0)k(0)(1-u)(1-\ell) (pg(1) + (1-p) g(0))\\
&\phantom{xxx} + \mathbb E(g(\beta_{i,j})) \Big( h(1)k(1) u\ell + h(0)k(1)(1-u)\ell +h(1)k(0)u(1-\ell)\Big)\\
&=(pg(1) + (1-p) g(0)) \mathbb E( h(I^{(u)}_{i,j}) )\mathbb E(k(J^{(u)}_{i,j}))\\
&=\mathbb E( g(\alpha_{i-1,j-1}) ) \mathbb E( h(I^{(u)}_{i,j}) )\mathbb E(k(J^{(u)}_{i,j})). \qedhere
\end{align*}
\end{proof}

 The last necessary preliminary step is a corollary of Lemma \ref{burke} which generalises Lemma \ref{lem:drp-bus} by incorporating the random variables $\{ \alpha_{i-1, j-1}\}_{i,j \ge 1}$. To this effect, for any down-right path $\psi$ satisfying \eqref{eq:psi},  define the interior sites  $\mathcal I _{\psi}$ of $\psi$ to be 
\be \label{eq:interior}
\mathcal I _{\psi} = \{ w \in \Z^2_+: \exists\, v_i \in \psi \text{ s.t. } \,w < v_i \text{ coordinate-wise} \}. 
\ee
Then
\begin{corollary} \label{cor:downrightpath}
	Fix a down-right path $\psi$ and recall definitions \eqref{eq:patheses}, \eqref{eq:interior}. The random variables 
	\[
	\{ \{\alpha_w\}_{w \in \mathcal I _{\psi}}, \{L_{\psi_i}\}_{i \in \Z} \}
	\]
	are mutually independent, with marginals 
	\[
	\alpha_{w} \sim \textrm{Ber(p)}, \quad  L_{\psi_i} \sim 
\begin{cases}
	\textrm{Ber}(u), & \textrm{if } \psi_i = e_1\\
	\textrm{Ber}\Big( \frac{p(1-u)}{ u + p(1-u) } \Big), & \text{if } \psi_i = -e_2. 
\end{cases}
	\]
\end{corollary} 
The proof is similar to that of Lemma \ref{lem:drp-bus} and we omit it.

\subsection{Law of large numbers for the boundary model.}

\begin{proof}[Proof of Theorem \ref{thm:LLNG}] 
From equations \eqref{eq:gradients} we can write 
\[
G^{(u)}_{\fl{Ns}, \fl{Nt}} = \sum_{j=1}^{\fl{Nt}}J^{(u)}_{0, j} + \sum_{i=1}^{\fl{Ns}}I^{(u)}_{i, \fl{Nt}} 
\]
since the $I, J$ variables are increments of the passage time. By the definition of the boundary model, the variables are i.i.d.\ Ber$(p(1-u)/(u+p(1-u))$. Scaled by $N$, the first sum converges to $t \E(J_{0,1})$ by the law of large numbers.

By Corollary \ref{cor:downrightpath} they are i.i.d.\ Ber$(u)$, since they belong on the down-right path that follows the vertical axes from $\infty$ down to $(0, \fl{Nt})$ and then moves horizontally. We cannot immediately appeal to the law of large numbers as the whole sequence changes with $N$ so we first appeal to the Borel-Cantelli lemma via a large deviation estimate. Fix an $\e>0$. 
\begin{align*}
\P\Big\{ N^{-1} \sum_{i=1}^{\fl{Ns}}I^{(u)}_{i, \fl{Nt}} \notin (u-\e, u+ \e) \Big\} &=\P\Big\{ N^{-1} \sum_{i=1}^{\fl{Ns}}I^{(u)}_{i, 0} \notin ( su-\e, su+ \e) \Big\} \\
& \le e^{-c(u, s ,\e)N},
\end{align*}
for some appropriate positive constant $c(u, s, \e)$. By the Borel-Cantelli lemma we have that for each $\e>0$ there exists a random $N_{\e}$ so that for all $N > N_\e$
\[
su - \e < N^{-1}\sum_{i=1}^{\fl{Ns}}I^{(u)}_{i, \fl{Nt}} \le su + \e.
\]
Then we have 
\[
su + t\, \frac{p(1-u)}{ u + p(1-u) }  - \e \le \varliminf_{N \to \infty} \frac{G^{(u)}_{\fl{Ns}, \fl{Nt}}}{N} \le \varlimsup_{N \to \infty} \frac{G^{(u)}_{\fl{Ns}, \fl{Nt}}}{N} \le su + t\, \frac{p(1-u)}{ u + p(1-u) }  + \e.
\]
Let $\e$ tend to 0 to finish the proof. 
\end{proof}

\subsection{ Law of large numbers for the i.i.d.\ model}

\begin{proof}[Proof of Theorem \ref{thm:LLNp}]

Let $g_{pp}(s,t) = \lim_{N \to \infty} N^{-1} G_{\fl{Ns}, \fl{Nt}}$ and denote by $g_{pp}^{(u)}(s,t) =  \lim_{N \to \infty} N^{-1} G^{(u)}_{\fl{Ns}, \fl{Nt}}$.
Recall that $g_{pp}(s,t)$ is 1-homogeneous and concave. 

The starting point is equation \eqref{eq:varform}. Scaling that equation by $N$ gives us the macroscopic variational formulation 
\begin{align}
g_{pp}^{(u)}&(1,1) \notag\\
&=\sup_{0\leq z\leq1}\{g_{pp}^{(u)}(z,0)+g_{pp}(1-z,1)\}\bigvee\sup_{0\leq z\leq1}\{g_{pp}^{(u)}(0, z)+g_{pp}(1,1-z)\} \notag\\
&= \sup_{0\leq z\leq1}\{z \E(I^{(u)})+g_{pp}(1-z,1)\}\bigvee\sup_{0\leq z\leq1}\{ z \E(J^{(u)})+g_{pp}(1,1-z)\}.\label{eq:maxpath}
\end{align}

We postpone this bit of the proof until the end. Assume \eqref{eq:maxpath} holds. We treat the case for which the first supremum in \eqref{eq:maxpath} is larger than the second. Since we are studying a symmetric model, this assumption is equivalent to consider $t\geq s$. We obtain the other case performing the transformation $t=s$.
Abbreviate $g_{pp}(1-z,1) = \psi(1-z)$ and substitute $ g_{pp}^{(u)}(1,1) $ using Theorem \ref{thm:LLNG}. Then 
\be\label{eq:13}
u+\frac{p(1-u)}{ u + p(1-u) }=\sup_{0\leq z\leq1}\{zu+\psi(1-z)\}.
\ee
Set $x = 1 -z$. $x$ still ranges in $[0, 1]$ and after a rearrangement of the terms, we obtain
\be
-\frac{p(1-u)}{ u + p(1-u) }=\inf_{0\leq x\leq1}\{xu-\psi(x)\}.
\ee
The expression on the right-hand side is the Legendre transform of $\psi$, and we have that its concave dual $\psi^{*}(u)= -\frac{p(1-u)}{ u + p(1-u) }$ with $u\in(0,1]$.
Since $\psi(x)$ is concave, the Legendre transform of $\psi^*$ will give back $\psi$, i.e. $\psi^{**} = \psi$. Therefore, 
\begin{align}
g_{pp}(x, 1) = \psi(x)&=\psi^{**}(x)=\inf_{0<u \le1}\{xu-\psi^*(u)\}=\inf_{0<u\leq1}\Big\{xu+\frac{p(1-u)}{ u + p(1-u) }\Big\}\notag\\
&= \inf_{0<u\leq1}\big\{x\E(I^{(u)})+\E(J^{(u)})\big\}, \quad \text{ for all } x \in [0, 1]. \label{eq:14}
\end{align}
Since $g_{pp}(s,t) = t g_{pp}(st^{-1}, 1)$, the first equality in \eqref{eq:busopt} is now verified. 
For the second equality, we solve the variational problem \eqref{eq:14}. The derivative of the expression in the braces has a critical point 
$u^* \le 1$ only when $p<x < 1$. In that case,  the infimum is achieved at 
\[
u^* =\frac{1}{1-p}\Big(\sqrt{\frac{p}{x}}-p\Big)
\]
and $g_{pp}(x, 1) =1/(1-p)[2\sqrt{xp}-p(x+1)]$. Otherwise, when $x \le p$ the derivative in $(0, 1)$ is always negative and the minimum occurs at $u^* = 1$. This gives
$g_{pp}(x, 1) = x$. Again, extend to all  $(s, t)$ via the relation $g_{pp}(s,t) = t g_{pp}(st^{-1}, 1)$. This concludes the proof for the explicit shape under \eqref{eq:maxpath} which we now prove.

For a lower bound, fix any $z \in [0, 1]$. Then 
\[
G^{(u)}_{N, N} \ge \sum_{i=1}^{\fl{Nz}} I^{(u)}_{i, 0} + G_{(\fl{Nz}, 1), (N, N)}. 
\]
Divide through by $N$. The left hand side converges  a.s.\ to $g^{(u)}_{pp}(1,1)$. the first term on the right converges a.s.\ to $z \E(I^{u})$. The second on the right, converges in probability to the constant $g_{pp}(1-z, 1)$. In particular, we can find a subsequence $N_k$ such that the convergence is almost sure for the second term. Taking limits on this subsequence, we conclude 
\[
g_{pp}^{(u)}(1,1) \ge z \E(I^{(u)}) + g_{pp}(1-z, 1).
\]  
Since $z$ is arbitrary we can take supremum over  $z$ in both sides of the equation above. 
The same arguments will work if we move on the vertical axis. Thus, we obtain the lower bound for \eqref{eq:maxpath}. 
For the upper bound, fix $\e>0$ and let $\{ 0 =q_0, \e=q_1, 2\e=q_2, \ldots, \fl{\e^{-1}}\e,  1=q_M \}$ a partition of $(0,1)$. We partition both axes. 
The maximal path that utilises $G^{(u)}_{N,N}$ has to exit between $\fl{Nk\e}$ and $\fl{N(k+1)\e}$ for some $k$. Therefore,  
we may write 
\begin{align*}
G^{(u)}_{N, N} &\le \max_{0 \le k \le \fl{\e^{-1}}}\bigg\{\sum_{i=1}^{\fl{N(k+1)\e}} I^{(u)}_{i, 0} + G_{(\fl{Nk\e}, 1), (N, N)}\bigg\}\\
&\phantom{xxxxxxxxxxx}\bigvee \max_{0 \le k \le \fl{\e^{-1}}}\bigg\{\sum_{j=1}^{\fl{N(k+1)\e}} J^{(u)}_{0, j} + G_{1, (\fl{Nk\e}), (N, N)}\bigg\}.
\end{align*}
Divide by $N$. The right-hand side converges in probability to the constant 
\begin{align*}
 &\max_{0 \le k \le \fl{\e^{-1}}}\{(k+1)\e u + g_{pp}(1-\e k, 1)\}\\
&\phantom{xxxxxxx}\bigvee \max_{0 \le k \le \fl{\e^{-1}}}\bigg\{(k+1)\e\frac{p(1-u)}{ u + p(1-u) }+ g_{pp}(1, 1-\e k) \bigg\}\\
&=\max_{q_k }\{q_k u + g_{pp}(1-q_k, 1)\} + \e u\\
&\phantom{xxxxxxx}\bigvee \max_{q_k}\bigg\{q_k \frac{p(1-u)}{ u + p(1-u) }+ g_{pp}(1, 1-q_k) \bigg\} +\e\frac{p(1-u)}{ u + p(1-u) }\\
&\le \sup_{0 \le z \le 1}\{z u + g_{pp}(1-z, 1)\} + \e u\\
&\phantom{xxxxxxx}\bigvee \max_{0\le z \le 1}\bigg\{z \frac{p(1-u)}{ u + p(1-u) } + g_{pp}(1, 1- z) \bigg\} +\e\frac{p(1-u)}{ u + p(1-u) }.
\end{align*}
The convergence becomes a.s. on a subsequence. The upper bound for \eqref{eq:maxpath} now follows by letting $\e \to 0$ in the last inequality.   
\end{proof}

In the following sections, either when the explicit dependence on $u$ is not important  or when it is not necessary and there will be no confusion, we omit the superscripts $(u)$ from the gradients $I, J$ without a particular mention. 

 \section{Upper bound for the variance in characteristic directions} 
\label{sec:varub}
 We follow the approach of \cite{balazs2006cube, Sep-12} in order to find  the order of the variance. All the difficulties and technicalities in our case arise from two facts: First that the random variables are discrete and small perturbations in the distribution do not alter the value of the random weight. Second, we have three potential steps to contest with rather than then usual two.
 
%

 
\subsection{The variance formula}

Let $(m,n)$ be a generic lattice site. Eventually we will define how $m, n$ grow to infinity using the parameter $u$. Define the passage time increments (labelled by compass directions) by 
\[
\cW=G_{0,n}^{(u)}-G_{0,0}^{(u)},\quad\cN=G_{m,n}^{(u)}-G_{0,n}^{(u)},\quad\cE=G_{m,n}^{(u)}-G_{m,0}^{(u)},\quad\cS=G_{m,0}^{(u)}-G_{0,0}^{(u)}.
\]
From Corollary \ref{cor:downrightpath}  we get that each of $\cW, \cN, \cE$ and $\cS$ is a sum of i.i.d.\ random variables and most importantly, $\cN$ 
is independent of $\cE$ and $\cW$ is independent of $\cS$ by the definition of the boundary random variables. From the definitions it is immediate to show the cocycle property for the whole rectangle $[m]\times [n]$ 
\be \label{eq:coprop}
\cW + \cN = G^{(u)}_{m,n} = \cS + \cE. 
\ee

We can immediately attempt to evaluate the variance of $G^{(u)}_{m,n}$ using this relations, by
\begin{align}
\Var(G^{(u)}_{m,n})&=\Var(\cW+\cN)\nonumber\\
&=\Var(\cW)+\Var(\cN)+2\Cov(\cS+\cE-\cN,\cN)\notag\\
&=\Var(\cW)-\Var(\cN)+2\Cov(\cS,\cN), \label{eq:var1}
\end{align}
Equivalently, one may use $\cE$ and $\cS$ to obtain 
\be \label{eq:var2}
\Var(G^{(u)}_{m,n}) =\Var(\cS)-\Var(\cE)+2\Cov(\cE,\cW).
\ee

In the sequence, when several Bernoulli parameters will need to be considered simultaneously, will add a superscript $(u)$ on the quantities $\cN, \cE, \cW, \cS $   to explicitly denote dependance on parameter $u$.

The covariance is not an object that can be computed directly, so the biggest proportion of this subsection is dedicated in finding a different way to compute the covariance that also allows for estimates and connects fluctuations of the maximal path with fluctuations of the last passage time. 

In the exponential exactly solvable model there is a clear expression for the covariance term. Unfortunately this does not happen here, so we must estimate the order of magnitude. This is the content of the next proposition.

\begin{proposition}\label{lem:exit} Fix $0 < u \le 1$. 
There exist functions $A_{\cN^{(u)}}$, $A_{\cE^{(u)}}$ such that for any $m,n \in \N$ we have

\be
\label{eq:var3}
\begin{aligned}
\Var(G_{m,n}^{(u)})&= n\frac{pu(1-u)}{[u+p(1-u)]^2}-mu(1-u)+2u(1-u)A_{\cN^{(u)}}\\
&= mu(1-u)-n\frac{pu(1-u)}{[u+p(1-u)]^2}-2u(1-u) A_{\cE^{(u)}}.
\end{aligned}
\ee
\end{proposition}
The result is proved by perturbing the parameter on one of the boundaries. Throughout the proof, the endpoint $(m,n)$ and the parameter $u$ are fixed. Pick an $\e>0$ and define a new parameter $u_\e$ so that $  u_\e = u+ \e  < 1$.  The only way this is not possible is when $u = 1$. If that's the case, $G^{(1)}_{m,n} = m$ is deterministic and the variance is zero. Equation \eqref{eq:var3} remains true as the right-hand side is a multiple of $(1-u)$.

For any fixed realization of $\om^{(u)} =  \{ \om^{(u)}_{i,0}, \om^{(u)}_{0,j}, \om^{(u)}_{i,j} \}$ with marginal distributions \eqref{eq:omu} we use the parameter $\e$ to modify the weights on the south boundary only. 
Define new bernoulli weights $\om^{\lambda_\e}$ via the conditional distributions
\begin{align}
&\P\{ \om^{u_\e}_{i,0} = 1 | \om^{(u)}_{i,0} = 1 \} = 1, \notag \\
&\P\{ \om^{u_\e}_{i,0} = 1 | \om^{(u)}_{i,0} = 0 \} = \frac{\e}{1-u}, \label{eq:bincoup}\\
&\P\{ \om^{u_\e}_{i,0} = 0 | \om^{(u)}_{i,0} = 0 \} = 1 -\frac{\e}{1-u}, \notag
\end{align}
i.e. we go through the values on the south boundary, and given the environment returned a 0, we change the value to a 1 with probability $\e$. Then $\{ \om^{u_\e}_{i,0} \}_{1 \le i \le m}$ is a collection of independent Ber$(u_\e)$ r.v. It is convenient to introduce an algebraic mechanism to construct $\om^{u_\e}$ directly. To this effect introduce a sequence of independent Bernoulli random variables $\mathscr H^{(\e)}_i \sim \text{Ber}\big(\frac{\e}{1-u} \big)$, $1 \le i \le m$ that are independent of the $\om^{(u)}$. Denote their joint distribution by $\mu_\e$. Then construct $\om^{u_\e}$ the following way: 
\be \label{eq:om-B}
\om^{u_\e}_{i,0} =  \mathscr H^{(\e)}_i \vee \om^{(u)}_{i,0} .
\ee
Check that \eqref{eq:om-B} satisfies \eqref{eq:bincoup}. It also follows that  
\be\label{eq:Berbnd}
\om^{u_\e}_{i,0}  - \om^{(u)}_{i,0} \le  \mathscr H^{(\e)}_i.
\ee
Under this modified environment, 
\be\label{eq:omper}
\om^{u_\e}_{i, 0} \sim \textrm{Ber}(u_\e), \quad \om^{(u)}_{i,j} \sim \textrm{Ber}(p), \quad \om^{(u)}_{0,j}\sim\textrm{Ber}\Big( \frac{p(1-u)}{ u + p(1-u) } \Big), 
\ee
the passage time is denoted by $G^{u_\e}$ and when we are referring to quantities in this model we will distinguish them with a superscript $u_\e$. With these definitions we have $\cS^{u_\e} \sim \text{Bin}(m, u+ \e)$, with mass function denoted by $f_{\cS^{u_\e}}(k) = \P\{ \cS^{u_\e} = k\}$, $0 \le k \le m$. 

Similarly, there will be instances for which we want to perturb  only the weights of the vertical axis, again when the parameter will change from  $u$ to $u+\e$. 
In that case, we denote the modified environment by $\cW^{u_\e}$ and it is given by 
\be\label{eq:omper2}
\om^{(u)}_{i, 0} \sim \text{Ber}(u), \quad \om^{(u)}_{i,j} \sim \text{Ber}(p), \quad \om^{u_\e}_{0,j}\sim \text{Ber}\Big( \frac{p(1-u-\e)}{ u+\e + p(1-u-\e) } \Big), 
\ee

Again, we use auxiliary i.i.d.\ Bernoulli variables $\{ \mathscr V^{(\e)}_j \}_{1 \le j \le n}$ with 
\[
 \mathscr V_j^{(\e)} \sim \text{Ber}\Big(1 - \e \frac{1 + u(1-p)}{(1-u)(p + u(1-p)) + (1-p)\e}\Big), 
\]
where we assume that $\e$ is sufficiently small so that the distributions are well defined. Then, the perturbed  weights on the vertical axis are defined by 
\be\label{eq:ammcs}
\om^{u_\e}_{0,j} = \om^{(u)}_{0,j}\cdot \mathscr V_j^{(\e)}.
\ee
Denote by $\nu_\e$ the joint distribution of $\mathscr V_j^{(\e)}$.
It will also be convenient to couple the environments with different parameters. In that case we use common realizations of i.i.d.\ Uniform$[0,1]$ random variables $\eta = \{\eta_{i,j}\}_{(i,j) \in \Z^2_+}$. The Bernoulli environment in the bulk is then defined as
\[
\om_{i,j} = \mathbbm1\{ \eta_{i,j} < p\}
\]  
and similarly defined for the boundary values. The joint distribution for the uniforms we denote by $\P_\eta$.

\begin{proposition}\label{prop:exitbound2}
The following bounds in terms of the right-most exit points of the maximal paths hold  
\be
A_{\cN^{(u)}}= \begin{cases}
\Cov(\cS^{(u)}, \cN^{(u)}) = \displaystyle\lim_{\e \to 0}\frac{ \E_{\P \otimes\mu_\e }( \cN^{u_\e} - \cN^{(u)})}{\e}, & 0<u<1\\
0& u=0,1.
\end{cases}
\ee
Similarly, 
\be
A_{\cE^{(u)}}= \begin{cases}
\Cov(\cW^{(u)}, \cE^{(u)}) = \displaystyle\lim_{\e \to 0}\frac{ \E_{\P \otimes\nu_\e }( \cE^{u_\e} - \cE^{(u)})}{\e}, & 0<u<1\\
0& u=0,1.
\end{cases}
\ee
\end{proposition}

\begin{proof}[Proof of Proposition \ref{lem:exit}]

The conditional joint distribution of $(\om^{u_\e}_{i,0})_{1 \le i \le m}$ given the value of their sum $\cS^{u+\e}$ is independent of the parameter $\e$. This is because the sum of i.i.d.\ Bernoulli is a sufficient statistic for the parameter of the distribution. In particular this implies that 
$
E[ \cN^{u+\e} | \cS^{u+\e}  = k]$ $= E_{\P \otimes\mu_\e }[ \cN^{(u)} | \cS^{(u)} = k]. 
$
For clarity, we added the superscript $(u)$ on the background measure $\P$ to emphasise that it is the measure on environment  $\om^{(u)}$.

Then we can compute the $\E(\cN^{u+\e})$
\begin{align}
\E_{\P \otimes\mu_\e }( \cN^{u_\e} - \cN^{(u)})&= \sum_{k=0}^m E[\cN^{u_\e} | \cS^{u_\e}=k] \P\otimes\mu_\e\{ \cS^{u_\e} = k\} - \E_{\P}(\cN^{(u)}) \notag\\
&=\sum_{k=0}^mE[\cN^{(u)}| \cS^{(u)}=k]\P\otimes\mu_\e\{ \cS^{u_\e} = k\}  - \E_{\P}(\cN^{(u)})\notag\\
&= \sum_{k=0}^mE[\cN^{(u)}| \cS^{(u)}=k]\big( \P\otimes\mu_\e\{ \cS^{u_\e} = k\}  - \P\{ \cS^{(u)} = k\} \big) \label{eq:northdiff}
\end{align}
To show that the limits in the statement are well defined, it suffices to compute
\begin{align*}
\lim_{\e \to 0}& \frac{ \P\otimes\mu_\e\{ \cS^{u_\e} = k\}  - \P\{ \cS^{(u)} = k\} }{\e} \\
&= {m \choose k} \lim_{\e \to 0}\frac{(u+\e)^k(1 - u -\e)^{m-k} - u^k(1 - u)^{m-k}}{\e}\\
&= {m \choose k} \frac{d}{du} u^k(1 - u)^{m-k} = {m \choose k} \frac{k - mu}{u(1-u)} u^k(1 - u)^{m-k}.
\end{align*}
Combine this with \eqref{eq:northdiff} to obtain
\begin{align}
\displaystyle\lim_{\e \to 0}\frac{ \E_{\P \otimes\mu_\e }( \cN^{u_\e} - \cN^{(u)})}{\e}  &= \frac{1}{u(1-u)} \sum_{k=0}^m E[\cN^{(u)} | \cS^{(u)}=k] k \P\{ \cS^{(u)} =k\}\notag \\
&\phantom{xxxxxxx} -  \frac{mu}{u(1-u)} \sum_{k=0}^m E[\cN^{(u)} | \cS^{(u)}=k] \P\{ \cS^{(u)}=k\} \notag \\
&= \frac{1}{u(1-u)}\Big(\E(\cN^{(u)}\cS^{(u)}) - \E(\cN^{(u)})\E(\cS^{(u)})\Big) \notag\\
&= \frac{1}{u(1-u)}\Cov(\cN^{(u)}, \cS^{(u)}). \label{eq:covV1}
\end{align}
Identical symmetric arguments, prove the remaining part of the proposition.
\end{proof}

For the rest of this proof, we prove the estimates on $A_{\cN^{(u)}}$ by estimating the covariance in a different way.

Fix any boundary site $w = (w_1, w_2) \in \{ (i,0), (0, j): 1 \le i \le m, 1\le j \le n\}$. The total weight in environment $\om$ collected on the boundaries by a path that exits from the axes at $w$ is  
\be \label{eq:bndw}
\mathscr{S}_{w} = \sum_{i=1}^{w_1} \om_{i,0} +   \sum_{j=1}^{w_2} \om_{0, j},
\ee
where the empty sum takes the value $0$. Let $\mathscr S^{(u)}$ be the above sum in environment $\om^{(u)}$ and let 
$\mathscr S^{u_e}$ denote the same, but in environment \eqref{eq:omper}.

Recall that $\xi_{e_1}$ is the rightmost exit point of any potential maximal path from the horizontal boundary, since it is the exit point of the right-most maximal path.  Similarly, if $\xi_{e_2} > 0$ the it is the down-most possible exit point. When the dependence on the parameter $u$ is important, we put superscripts $(u)$ to denote that.

\begin{lemma} Let $\xi_{e_1}$ be the exit point of the maximal path in environment $\om^{(u)}$. Let $\cN^{u_\e}$ denote the last passage increment in environment \eqref{eq:omper} of the north boundary and $\sS_{\xi_{e_1}}^{u_\e}$ the weight collected on the horizontal axis in the same environment, but only up to the exit point of the maximal path in environment $\om^{(u)}$. $\cN^{(u)}$ $\sS_{\xi_{e_1}}^{(u)}$ are the same quantities in environment $\om^{(u)}$. Then
\be \label{eq:exn-n1}
\E_{\P\otimes\mu_\e}(\sS_{\xi_{e_1}}^{u_\e} - \sS^{(u)}_{\xi_{e_1}}) \le \E_{\P\otimes\mu_\e}(\cN^{u_\e} - \cN^{(u)} ) \le \E_{\P\otimes\mu_\e}(\sS_{\xi_{e_1}}^{u_\e} - \sS^{(u)}_{\xi_{e_1}}) + C(m,u,p) \e^{3/2}.
\ee

Similarly, in environments \eqref{eq:omper2} and $\om^{(u)}$,
\be \label{eq:exn-m}
\E_{\P\otimes\nu_\e}(\sS_{\xi_{e_2}}^{u_\e} - \sS^{(u)}_{\xi_{e_2}}) \ge \E_{\P\otimes\nu_\e}(\cE^{u_\e} - \cE^{(u)} ) \ge \E_{\P\otimes\nu_\e}(\sS_{\xi_{e_2}}^{u_\e} - \sS^{(u)}_{\xi_{e_2}}) - C(n,u,p) \e^{4/3}.
\ee
\end{lemma}

\begin{proof}
We only prove \eqref{eq:exn-m}. Same arguments work for \eqref{eq:exn-n1}. Modify the weights on the vertical axis and create environment $\cW^{u_\e}$ given by \eqref{eq:omper2}. 
The first inequality in equation \eqref{eq:exn-n} follows by first noting that
\be \label{eq:ommm}
\cE^{u_\e} - \cE^{(u)}   \le \cW^{u_\e} - \cW \le 0. 
\ee
The left inequality in \eqref{eq:exn-m} is then immediate, because the modification decreases all weights on the west boundary by \eqref{eq:ammcs}. 
To see equation \eqref{eq:ommm}, do a double induction on $m, n$ using equations \eqref{eq:4} and the cocycle property \eqref{eq:coprop}, starting from the first corner square. 

The remaining proof is to establish the second inequality in \eqref{eq:exn-m}. 
Consider the event $\{ \xi^{u_{\e}} \neq \xi \}$. Since we only modify weights on the vertical axis, the exit point $\xi$ of the original maximal path 
will be different from $\xi^{u_\e}$ only if $\xi^{u_\e} = \xi^{u_\e}_{e_2}$. Moreover, since the modification actually decreases the weights, one of two things may happen:
\begin{enumerate}
\item $\xi^{u_\e} \neq \xi$ and $\sS^{u_\e}_{\xi^{u_\e}_{e_2}} + G_{(1, \xi^{u_\e}_{e_2}),(m,n)} > \sS^{u_\e}_{\xi_{e_2}} + G_{(1, \xi_{e_2}), (m,n)}$, or
\item 
  $\xi^{u_\e} \neq \xi$ and $\sS^{u_\e}_{\xi^{u_\e}_{e_2}} + G_{(1, \xi^{u_\e}_{e_2}),(m,n)} = \sS^{u_\e}_{\xi_{e_2}} + G_{(1, \xi_{e_2}), (m,n)}$ 
\end{enumerate}

We use these cases to define two auxiliary events: 
\begin{align*}
\mathcal{A}_{1}&= \{ \xi^{u_\e} \neq \xi \text{ and } \sS^{u_\e}_{\xi^{u_\e}_{e_2}} + G_{(1, \xi^{u_\e}_{e_2}),(m,n)} > \sS^{u_\e}_{\xi_{e_2}} + G_{(1, \xi_{e_2}), (m,n)} \},\\
\mathcal{A}_{2}&= \{   \xi^{u_\e} \neq \xi \text{ and } \sS^{u_\e}_{\xi^{u_\e}_{e_2}} + G_{(1, \xi^{u_\e}_{e_2}),(m,n)} = \sS^{u_\e}_{\xi_{e_2}} + G_{(1, \xi_{e_2}), (m,n)} \}
\end{align*}
and note that $
\{ \xi^{u_{\e}} \neq \xi \} = \mathcal{A}_{1}\cup\mathcal{A}_{2}$. On $\mathcal{A}_2$ we can bound
\begin{align*}
\cE^{u_\e} - \cE^{(u)}  &= G^{u_\e}_{m,n} -  G^{(u)}_{m,n} =  \sS^{u_\e}_{\xi_{e_2}} + G_{(1, \xi_{e_2}), (m,n)} -  \sS^{(u)}_{\xi_{e_2}} - G_{(1, \xi_{e_2}), (m,n)}\\
&= \sS^{u_\e}_{\xi_{e_2}} - \sS^{(u)}_{\xi_{e_2}}.
\end{align*}
Then we estimate 
\begin{align}
\cE^{u_\e} - \cE^{(u)}  &= (\cE^{u_\e}-\cE^{(u)} )\cdot \mathbbm{1} \{\xi^{u_\e}= \xi \}+(\cE^{u_\e}-\cE^{(u)} )\cdot \mathbbm{1}\{\xi^{u_\e}\neq \xi\}\notag \\
&=(\sS_\xi^{u_\e}-\sS^{(u)}_\xi)\cdot \mathbbm{1}\{\xi^{u_\e}= \xi \}+(\cE^{u_\e}-\cE^{(u)} )\cdot (\mathbbm{1}_{\mathcal A_1} + \mathbbm{1}_{\mathcal A_2} )\label{eq:hm0} \\
&\ge (\sS^{u_\e}_{\xi_{e_2}}-\sS^{(u)}_{\xi_{e_2}})+(\cE^{u_\e}-\cE^{(u)} )\cdot \mathbbm{1}_{\mathcal A_1}. \label{eq:n-n}
\end{align}
The last inequality is justified the following way: Only the weights on the vertical axis were changed-actually decreased. Therefore, if the maximal path chose to move horizontally before the modification, it would do so after and the first term in \eqref{eq:hm0} must be 0. The first term may not equal zero only when the maximal path takes a vertical first step before the modification. On the event  $\mathbbm{1}\{\xi^{u_\e}= \xi \}$ the bound in \eqref{eq:n-n} still holds.

To bound the second term of \eqref{eq:n-n}, we use H\"older's inequality with exponents $p = 3, q = 3/2$ to obtain 
\be \label{eq:pxibound}
\E_{\P\otimes\nu_\e}((\cE^{(u)} -\cE^{u_\e})\cdot \mathbbm{1}_{\mathcal A_1})\le \E_{\P\otimes\nu_\e}((\cE^{(u)} -\cE^{u_\e})^3)^{1/3}({\P\otimes\nu_\e}\{ \mathcal A_1\})^{2/3}.
\ee
The first expectation on the right is bounded above $C(u,p)n$ since $\cE^{(u)} $ is a sum of i.i.d.\ Bernoulli random variables that bounds above $\cE^{(u)} -\cE^{u_\e}$.

Now to bound the probability. 
 Consider the equality of events 
\begin{align*}
\mathcal{A}_{1}&= \{  \sS^{u_\e}_{k} + G_{(1,k),(m,n)} > \sS^{u_\e}_{\xi_{e_2}} + G_{(1, \xi_{e_2}), (m,n)} \text{ for some } 0\le k\neq \xi \le n \}\\
&= \{  \sS^{u_\e}_{k} - \sS^{u_\e}_{\xi_{e_2}} >  G_{(1, \xi_{e_2}), (m,n)} -G_{(1,k),(m,n)} \text{ for some } 0\le k\neq \xi \le n \}\\
&= \{  \sS^{u_\e}_{k} - \sS^{u_\e}_{\xi_{e_2}} >  G_{(1,\xi_{e_2}), (m,n)} -G_{(1,k),(m,n)} \ge  \sS^{(u)}_{k} - \sS^{(u)}_{\xi_{e_2}} \\
&\phantom{xxxxxxxxoxxxxxxxxxxxxxoxxxxxoxxxxxx}\text{ for some } 0\le k\neq \xi \le n \}.
\end{align*}
Coupling \eqref{eq:ammcs} implies that the events above are empty when $k > \xi_{e_2}$. Therefore, consider the case $\xi_{e_2} > k$. In that case, since $\xi_{e_2}$ is the down-most possible exit point, the second inequality in the event above can be strict as well. Thus 
\[
\mathcal A_1 \subseteq \bigcup_{(k,i): 0 \le k <  i \le n} \{  \sS^{u_\e}_{k} - \sS^{u_\e}_{i} >  G_{(1,i), (m,n)} -G_{(1,k),(m,n)} >  \sS^{(u)}_{k} - \sS^{(u)}_{i} \}.
\] 
The strict inequalities in the event and the fact that these random variables are integer, we see that the difference  
$\sS^{u_\e}_{k} - \sS^{u_\e}_{i} - \sS^{(u)}_{k} + \sS^{(u)}_{i} \ge 2$. This way, 
\begin{align}
2 &\le \sS^{u_\e}_{k} - \sS^{u_\e}_{i} - \sS^{(u)}_{k} + \sS^{(u)}_{i} = - \sum_{j=k+1}^i \om^{u_\e}_{0,j} + \sum_{j=k+1}^i \om^{(u)}_{0,j}\notag\\
&= \sum_{j=k+1}^i \Big(  \om^{(u)}_{0,j} - \om^{u_\e}_{0,j} \Big) \quad \text{by \eqref{eq:ammcs}}\notag\\
&\le  \sum_{j=0}^n \Big(  \om^{(u)}_{0,j} - \om^{u_\e}_{0,j} \Big) =   \sum_{j=0}^n   \om^{(u)}_{0,j}\big(1 - \mathscr V^{(\e)}_j \big) = \mathscr W_\e. \label{eq:halellu}
\end{align}
 $\mathscr W_\e$ is defined by the last equality above and we therefore just showed $\mathcal A_1 \subseteq \{ \mathscr W_\e \ge 2 \}$.

The event $\{ \mathscr W_\e \ge 2 \}$ holds if at least 2 indices $j$ satisfy with $  \om^{(u)}_{0,j}\big(1 - \mathscr V^{(\e)}_j \big)  = 1$. 
By definition \eqref{eq:halellu} we have that $\mathscr W_\e$ is binomially distributed with probability of success $C\e$ under $\P\otimes \nu_\e$ and therefore, in order to have at least two successes, 
\be
{\P\otimes\nu_\e}\{ \mathscr W_\e \ge 2 \} \le C(n, u) \e^2.  \label{eq:pxibound}
 \ee
Combine \eqref{eq:n-n} and \eqref{eq:pxibound} to conclude  
\be \label{eq:exn-n}
\E_{\P\otimes\nu_\e}(\cE^{u_\e} - \cE^{(u)} ) \ge \E_{\P\otimes\nu_\e}(\sS_{\xi_{e_2}}^{u_\e} - \sS^{(u)}_{\xi_{e_2}}) - C(n,u) \e^{4/3}. \qedhere
\ee

\end{proof}

\begin{lemma}\label{lem:Abound} 
Let $0 < u < 1$. Then, 
\be
A_{\cN^{(u)}} \le \frac{\E(\xi^{(u)}_{e_1})}{1-u}, \quad \text{ and } \quad A_{\cE^{(u)}}\ge - \frac{p(1+u(1-p))}{(u+p(1-u))^2}\E(\xi^{(u)}_{e_2})
\ee
\end{lemma}

\begin{proof}
Now we bound the first term. Compute 
\begin{align*}
 \E_{\P\otimes\nu_\e}(\sS_{\xi_{e_1}}^{u_\e} - \sS^{(u)}_{\xi_{e_1}}) &=\sum_{y = 1}^m E\Big[\sS_{y}^{u_\e} - \sS^{(u)}_{y}\Big | \xi_{e_1} = y\Big]\P\{ \xi_{e_1} = y\}\\
 & \le \sum_{y = 1}^m E\Big[\sum_{i=1}^y \mathscr H_i^{(\e)} \Big | \xi_{e_1} = y\Big]\P\{ \xi_{e_1} = y\}, \text{ from \eqref{eq:Berbnd}},\\
  & =\sum_{y = 1}^m \E_{\mu_\e}\Big[\sum_{i=1}^y \mathscr H_i^{(\e)} \Big]\P\{ \xi_{e_1} = y\}, \text{ since $\mathscr H_i$, $\om^{(u)}$ independent,}\\
  &= \e \frac{\E(\xi_{e_1})}{1-u}.
\end{align*}
Now substitute in \eqref{eq:exn-n}, divide through by $\e$ and take the limit as $\e \to 0$ to obtain 
\[
 \lim_{\e \to 0} \frac{\E(\cN^{\lambda_\e} - \cN)}{\e} \le \frac{\E(\xi^{(u)}_{e_1})}{1-u}. 
\]
For the second bound, write \allowdisplaybreaks
\begin{align*}
\E_{\P\otimes\nu_\e}&(\cE^{u+\e} - \cE^{(u)}) \ge \E_{\P\otimes\nu_\e}(\sS^{u+\e}_{\xi^{(u)}_{e_2}} - \sS^{(u)}_{\xi^{(u)}_{e_2}}) + o(\e)\\
&=\E_{\P\otimes\nu_\e}\Big(\sum_{j=1}^{\xi^{(u)}_{e_2}} \om^{u+\e}_{0,j} -\om^{(u)}_{0,j}\Big)+ o(\e) = - \E_{\P\otimes\nu_\e}\Big(\sum_{j=1}^{\xi^{(u)}_{e_2}} \om^{(u)}_{0,j}( 1- \mathscr V^{(\e)}_j)\Big)+ o(\e) \\
&= - \sum_{k=1}^{n} \sum_{j=1}^{k}  \E_{\P\otimes\nu_\e}\big(\om^{(u)}_{0,j}( 1- \mathscr V^{(\e)}_j) \mathbbm1\{ \xi_{e_2}^{(u)} = k \}\big)+ o(\e) \\
&= - \sum_{k=1}^{n} \sum_{j=1}^{k}  \E_{\P}\big(\om^{(u)}_{0,j}\mathbbm1\{ \xi_{e_2}^{(u)} = k \}\big) \E_{\nu_\e}(1- \mathscr V^{(\e)}_j)+ o(\e) \\
&\ge - \sum_{k=1}^{n} \sum_{j=1}^{k} \P\{ \xi_{e_2}^{(u)} = k \} E_{\nu_\e}(1- \mathscr V^{(\e)}_j)+ o(\e) \\
& = - \e \E_{\P}( \xi_{e_2}^{(u)}) \cdot \frac{(1+u(1-p))(p(1-u))}{(u+p(1-u))((1-u)(u + p(1-u)) +(1-p)\e)} + o(\e).
\end{align*}
Divide both sides of the inequality by $\e$ and let it tend to 0.
\end{proof}

\begin{lemma} Let $0 < r_1 < r_2 < 1$ and let $\xi^{(r_i)}$ the corresponding right-most  (resp. down-most) exit points for the maximal paths in environments coupled by common uniforms $\eta$. Then 
\[
\xi^{(r_1)}_{e_1} \le \xi^{(r_2)}_{e_1} \text{ and }\, \xi^{(r_1)}_{e_2} \ge \xi^{(r_2)}_{e_2}.
\]
\end{lemma}

\begin{proof}
Assume that in environment $\om^{(r_1)}$ the maximal path exits from the vertical axis. Then, since $r_2 > r_1$ and the weights coupled through common uniforms, realization by realization $\om^{(r_2)}_{0,j} \le \om^{(r_1)}_{0,j}$. 
Assume by way of contradiction that 
$
\xi^{(r_1)}_{e_2} < \xi^{(r_2)}_{e_2}. 
$
Then
\begin{align*}
G_{(1, \xi^{(r_1)}_{e_2}),(m,n)} &\ge G_{(1, \xi^{(r_2)}_{e_2}),(m,n)} + \sS^{(r_1)}_{\xi^{(r_2)}_{e_2}}- \sS^{(r_1)}_{\xi^{(r_1)}_{e_2}}\\
&\ge G_{(1, \xi^{(r_2)}_{e_2}),(m,n)} + \sS^{(r_2)}_{\xi^{(r_2)}_{e_2}}- \sS^{(r_2)}_{\xi^{(r_1)}_{e_2}},
\end{align*}
giving
\[
G_{(0, \xi^{(r_1)}_{e_2}),(m,n)} + \sS^{(r_2)}_{\xi^{(r_1)}_{e_2}} \ge G_{(0, \xi^{(r_2)}_{e_2}),(m,n)} + \sS^{(r_2)}_{\xi^{(r_2)}_{e_2}} = G^{(r_2)}_{m,n},
\]
which cannot be true because $\xi^{(r_2)}_{e_2}$ is the down-most exit point in $\om^{(r_2)}$. The proof for a maximal path exiting the horizontal axis is similar.
\end{proof}
\subsection{Upper bound} 
In this section we prove the upper bound in Theorem  \eqref{thm:varub}. We begin with three comparison lemmas. One is for the two functions $A_{\cN^{(u)}}$ that appear in Proposition \ref{lem:exit} when we vary the parameter. The other comparison is between variances in environments with different parameters.  

\begin{lemma}\label{lem:covcomp}
Pick two parameters $0 < r_1 < r_2 < 1$. Then 
\be
\mathcal A_{\cN^{(r_1)}} \le \mathcal A_{\cN^{(r_2)}} + m(r_2 - r_1). 
\ee
\end{lemma}

\begin{proof}[Proof of Lemma \ref{lem:covcomp}] Fix an $\e > 0$ small enough so that $r_1 + \e < r_2$ and $r_2 + \e < 1$. This is not a restriction as we will let $\e$ tend to 0 at the end of the proof.  
We use a common realization of the Bernoulli variables $\mathscr H^{(\e)}_i$ and we couple the weights in the $\om^{(r_2)}$ and $\om^{(r_1)}$ environments using common uniforms $\eta = \{\eta_{i,j}\}$ (with law $\P_\eta$), 
independent of the $\mathscr H^{(\e)}_i$. 

We need to bound a different way starting from the line before \eqref{eq:n-n}.
\begin{align*}
\cN^{u+\e} - \cN^{(u)}&= (\sS_\xi^{u+\e}-\sS^{(u)}_\xi)\cdot \mathbbm{1}\{\xi^{u+\e}= \xi \}+(\cN^{u+\e}-\cN)\cdot (\mathbbm{1}_{\mathcal A_1} + \mathbbm{1}_{\mathcal A_2} )\notag \\
&= (\sS^{u+\e}_{\xi_{e_1}}-\sS^{(u)}_{\xi_{e_1}})\cdot \mathbbm{1}\{\xi^{u+\e}_{e_1}= \xi_{e_1} \}+(\cN^{u+\e}-\cN)\cdot (\mathbbm{1}_{\mathcal A_1} + \mathbbm{1}_{\mathcal A_2} )\notag \\
&= (\sS^{u+\e}_{\xi_{e_1}}-\sS^{(u)}_{\xi_{e_1}})+(\cN^{u+\e}-\cN - (\sS^{u+\e}_{\xi_{e_1}}-\sS^{(u)}_{\xi_{e_1}}))\cdot (\mathbbm{1}_{\mathcal A_1} + \mathbbm{1}_{\mathcal A_2}).\label{eq:w1}
\end{align*}

We first show that the second term can never be negative. 
Write 
\begin{align*}
\cN^{u+\e}-\cN &- (\sS^{u+\e}_{\xi_{e_1}}-\sS^{(u)}_{\xi_{e_1}}) = G^{u+\e}_{m,n}-G^{(u)}_{m,n} - (\sS^{u+\e}_{\xi_{e_1}}-\sS^{(u)}_{\xi_{e_1}})\\
&=\sS^{u+\e}_{\xi^{u+\e}_{e_1}} + G_{(\xi^{u+\e}_{e_1},1) (m,n)}- G_{(\xi_{e_1},1) (m,n)} - \sS^{u+\e}_{\xi_{e_1}}.
\end{align*}
On $\mathcal A_2$ this expression is $0$. On $\mathcal A_1$ the sum of the first two terms is
strictly larger than the sum of the last two. Then, $\eqref{eq:w1}$ becomes 
\[ 
\cN^{u+\e} - \cN^{(u)} \ge \sS^{u+\e}_{\xi_{e_1}}-\sS^{(u)}_{\xi_{e_1}}.
\]
 
Use this to bound the first term in the computation that follows. The second term we bound with  \eqref{eq:exn-n}. 
\begin{align*}
E_{\mu_\e \otimes\P_\eta}&(\cN^{r_2+\e}- \cN^{(r_2)})-E_{\mu_\e \otimes\P_\eta}(\cN^{r_1+\e} - \cN^{(r_1)})\\
&\ge E_{\mu_\e \otimes\P_\eta}(\sS^{r_2+\e}_{\xi^{(r_2)}_{e_1}} - \sS^{(r_2)}_{\xi^{(r_2)}_{e_1}})-E_{\mu_\e \otimes\P_\eta}(\sS^{r_1+\e}_{\xi^{(r_1)}_{e_1}} - \sS^{(r_1)}_{\xi^{(r_1)}_{e_1}}) + o(\e)\\
&=E_{\mu_\e \otimes\P_\eta}\Big(\sum_{i=1}^{\xi^{(r_2)}_{e_1}}\mathbbm{1}\{ \mathscr H_i^{(\e)}=1\}\mathbbm{1}\{ \eta_{i,0}>r_2\}\Big)\\
&\phantom{xxxxxxxxx}-E_{\mu_\e \otimes\P_\eta}\Big(\sum_{i=1}^{\xi^{(r_1)}_{e_1}}\mathbbm{1}\{ \mathscr H_i^{(\e)}=1\}\mathbbm{1}\{ \eta_{i,0}>r_1\}\Big)+ o(\e)\\
&\ge E_{\mu_\e \otimes\P_\eta}\Big(\sum_{i=1}^{\xi^{(r_1)}_{e_1}}\mathbbm{1}\{ \mathscr H_i^{(\e)}=1\}\Big(\mathbbm{1}\{ \eta_{i,0}>r_2\} - \mathbbm{1}\{ \eta_{i,0}>r_1\}\Big)\Big)+ o(\e)\\
&\ge - m E_{\mu_\e \otimes\P_\eta}\Big(\mathbbm{1}\{ \mathscr H_i^{(\e)}=1\}\Big(\mathbbm{1}\{ \eta_{i,0}>r_2\} - \mathbbm{1}\{ \eta_{i,0}>r_1\}\Big)\Big)+ o(\e)\\
&=-m\e(r_2 - r_1) + o(\e).
\end{align*}
Divide by $\e$ and let $\e \to 0$ to get the result.
\end{proof}

\begin{lemma}\label{lem:varcomp}[Variance comparison]
Fix $\delta_0 >0$ and parameters $u$, $r$  so that $p < p +\delta_0< u < r <1$. Then, there exists a constant 
$C = C(\delta_0, p) > 0$ so that for all admissible values of $u$ and $r$ we have
\be
\frac{\Var(G_{m,n}^{(u)})}{u(1-u)} \le \frac{\Var(G_{m,n}^{(r)})}{r(1-r)}+ C(m+n)(r-u).
\ee
\end{lemma}

\begin{proof} 
Begin from equation \eqref{eq:var3}, and bound
\allowdisplaybreaks
\begin{align*}
\frac{\Var(G_{m,n}^{(u)})}{u(1-u)}&= n\frac{p}{[u+p(1-u)]^2}-m+2A_{\cN^{(u)}}\\
&= n\frac{p}{[r+p(1-r)]^2}-m+2A_{\cN^{(u)}} + np(\frac{1}{[u+p(1-u)]^2} -\frac{1}{[r+p(1-r)]^2})\\
&\le \frac{\Var(G_{m,n}^{(r)})}{r(1-r) }+ np(\frac{1}{[u+p(1-u)]^2} -\frac{1}{[r+p(1-r)]^2})+ 2m(r-u)\\
&\le \frac{\Var(G_{m,n}^{(r)})}{r(1-r) }+ 2np(1-p)\frac{(r- u) }{[u+p(1-u)]^3} + 2m(r-u)\\
&\le \frac{\Var(G_{m,n}^{(r)})}{r(1-r) }+ 2n\frac{p(1-p)}{\delta_0^3}(r-u) + 2m(r-u).
\end{align*}
In the third line from the top we used Lemma \ref{lem:varcomp}. Set $C = 2\frac{p(1-p)}{\delta_0^3}\vee 2$ to finish the proof.
\end{proof}
From this point onwards we proceed by a perturbative argument. 
We introduce the scaling parameter $N$ that will eventually go to $\infty$ and the characteristic shape of the rectangle, given the boundary parameter. We will need to use the previous lemma, so we fix a $\delta_0 > 0$, so that $\delta_0< \lambda < 1$ and we choose a parameter $u = u(N,b,v) < \lambda$ so that 
\[ 
 \lambda -u = b\frac{v}{N}
\] 
At this point $v$ is free but $b$ is a constant so that $ \delta_0<  \lambda < u$ The north-east endpoint of the rectangle with boundary of parameter $\lambda$ is defined by $(m_\lambda(N), n_{\lambda}(N))$ which is the microscopic characteristic direction corresponding to $\lambda$ defined in \eqref{eq:bound}. 
%

The quantities $G_{(\xi_{e_2},1), (m,n)},\xi_{e_2}$ and $G_{m,n}$ connected to these indices are denoted by $G_{(\xi_{e_2},1), (m,n)}(N),\xi_{e_2}(N),G_{m,n}(N)$. In the proof we need to consider different boundary conditions and this will be indicated by a superscript. When the superscript $u$ will be used, the reader should remember that this signifies  changes on the boundary conditions and not the endpoint $(m_{\lambda}(N), n_{\lambda}(N))$, which will always be defined by \eqref{eq:bound} for a fixed $\lambda$. 

Since the weights $\{\om_{i,j}\}_{i,j\geq1}$ in the interior are not affected by changes in boundary conditions, the passage time $G_{(z,1), (m,n)}(N)$ will not either, for any $z < m_{\lambda}(N)$.

\begin{proposition} \label{yextibound}
Fix  $ \lambda \in (0,1)$. Then,  there exists a constant $K = K(\lambda, p)> 0$ so that for any $  b < K $, 
and $N$ sufficiently large
\be\label{eq:Z}
\P\{\xi_{e_2}^{(\lambda)}(N)>v\}\leq C \frac{N^2}{bv^3}\Big( \frac{\E(\xi^{(\lambda)}_{e_2})}{bv}+1\Big),
\ee
for all $v \ge 1$.
\end{proposition}

\begin{proof} We use an auxiliary parameter $u < \lambda$ so that 
\[
u = \lambda -  b v N^{-1} > 0.
\]
Constant $b$ is under our control.
We abbreviate  $(m_{\lambda}(N), n_{\lambda}(N))$ $= {\bf t}_N(\lambda)$. 
Whenever we use auxiliary parameters we explicitly mention it to alert the reader that the environments are coupled through common realizations of uniform random variables $\eta$. The measure that we are using for all computations is the background measure $\P_\eta$ but to keep the notation simple we omit the subscript $\eta$.
 
Since $G^{(u)}_{{\bf t}_N(\lambda)}(N)$ is utilised on the maximal path,
\[\sS^{(u)}_z+G_{(1,z),{\bf t}_N(\lambda)}(N)\leq G^{(u)}_{{\bf t}_N(\lambda)}(N)\]
for all $1\leq z\leq n_{\lambda}(N)$ and all parameters $p+\delta_0<u < \lambda <1$. Consequently, for integers $v\geq0$,
\begin{align}
\P\{\xi_{e_2}^{(\lambda)}(N)>v\}&=\P\{\exists \, z>v:\sS^{(\lambda)}_z+G_{(1,z),{\bf t}_N(\lambda)}(N)= G^{(\lambda)}_{{\bf t}_N(\lambda)}(N)\}\nonumber\\
&\leq\P\{\exists \,z>v:\sS^{(\lambda)}_z-\sS^{(u)}_z+G_{{\bf t}_N(\lambda)}^{(u)}(N)\geq G^{(\lambda)}_{{\bf t}_N(\lambda)}(N)\}\nonumber\\
&=\P\{\exists \, z>v:\sS^{(\lambda)}_z- \sS^{(u)}_z +G^{(u)}_{{\bf t}_N(\lambda)}(N)- G^{(\lambda)}_{{\bf t}_N(\lambda)}(N)\geq0\}\nonumber\\
&\leq\P\{\sS^{(\lambda)}_v-\sS^{(u)}_v+G^{(u)}_{{\bf t}_N(\lambda)}(N)- G^{(\lambda)}_{{\bf t}_N(\lambda)}(N)\geq0\}.\label{eq:prob}
\end{align}
The last line above follows from the fact that $u < \lambda$, which implies that $\sS^{(\lambda)}_k-\sS^{(u)}_k$ is non-positive and decreasing in $k$ when the weights are coupled through common uniforms. 
The remaining part of the proof goes into bounding the last probability above. For any $\alpha \in \R$ we further bound
\begin{align}
&\P\{\xi_{e_2}^{(\lambda)}(N)>v\}\leq\P\{\sS^{(\lambda)}_v-\sS^{(u)}_v\geq -\alpha\}\label{eq:prob1}\\
&\hspace{3cm}+\P\{G^{(u)}_{{\bf t}_N(\lambda)}(N)- G^{(\lambda)}_{{\bf t}_N(\lambda)}(N)\geq \alpha\}.\label{eq:prob2}
\end{align}
We treat \eqref{eq:prob1} and \eqref{eq:prob2} separately for  
\be
\alpha=-\E[\sS^{(\lambda)}_v-\sS^{(u)}_v]-C_0\frac{v^2}{N} 
\ee
where $C_0>0$. Restrictions on $C_0$ will be enforced in the course of the proof. 

\noindent \textbf{ Probability \eqref{eq:prob1}:} That is a sum of i.i.d.~random variables so we simply bound using Chebyshev's inequality.  The variance is estimated by 
\begin{align*}
\Var(\sS_v^{(\lambda)}-\sS^{(u)}_v)&= \sum_{j= 1}^v  \Var\big(\om^{(\lambda)}_{0,j} - \om^{(u)}_{0,j} \big) \le C_{p, \lambda} v(\lambda-u) = c_{p, \lambda}\frac{b v^2}{N}.
\end{align*}
Then by Chebyshev's inequality we obtain
\begin{align}
\P\Big\{\sS^{(\lambda)}_v-\sS^{(u)}_v\geq\E[\sS^{(\lambda)}_v-\sS^{(u)}_v]+C_0\frac{v^2}{N} \Big\}&\leq \frac{c_{\lambda, p}}{C_0^2}\cdot b\frac{N}{v^2}.
\end{align}

\noindent\textbf{ Probability \eqref{eq:prob2}: } 
Substitute in the value of $\alpha$ and subtract from both sides $\E[G^{(u)}_{{\bf t}_N(\lambda)}(N)- G^{(\lambda)}_{{\bf t}_N(\lambda)}(N)]$. Then
\begin{align}
\P&\{G^{(u)}_{{\bf t}_N(\lambda)}(N)- G^{(\lambda)}_{{\bf t}_N(\lambda)}(N)\geq \alpha\} \notag\\
&=\P\{G^{(u)}_{{\bf t}_N(\lambda)}(N)- G^{(\lambda)}_{{\bf t}_N(\lambda)}(N)-\E[G^{(u)}_{{\bf t}_N(\lambda)}(N)- G^{(\lambda)}_{{\bf t}_N(\lambda)}(N)]\notag\\
&\hspace{1.4cm}\geq v( \lambda - u)\frac{p}{(p + (1 -p)u)(p +(1- p)\lambda)}-C_0\frac{v^2}{N}-\E[G^{(u)}_{{\bf t}_N(\lambda)}(N)- G^{(\lambda)}_{{\bf t}_N(\lambda)}(N)]\}\notag\\
&\le\P\{G^{(u)}_{{\bf t}_N(\lambda)}(N)- G^{(\lambda)}_{{\bf t}_N(\lambda)}(N)-\E[G^{(u)}_{{\bf t}_N(\lambda)}(N)- G^{(\lambda)}_{{\bf t}_N(\lambda)}(N)]\notag\\
& \hspace{1.4cm}\ge  v( \lambda - u)C_{\lambda, p}-C_0\frac{v^2}{N}-\E[G^{(u)}_{{\bf t}_N(\lambda)}(N)- G^{(\lambda)}_{{\bf t}_N(\lambda)}(N)]\}\label{eq:xxxx}.
\end{align}
where
\[
C_{\lambda, p} = \frac{p}{(p+(1-p)\lambda)^2}.
\]
We then estimate
\allowdisplaybreaks 
\begin{align*}
\E[G^{(u)}_{{\bf t}_N(\lambda)}(N)&- G^{(\lambda)}_{{\bf t}_N(\lambda)}(N)] = m_{\lambda}(N)( u- \lambda) + n_{\lambda}(N)\Big( \frac{p(1-u)}{u + p(1-u)} - \frac{p(1-\lambda)}{\lambda + p(1-\lambda)}  \Big)\\
&=m_{\lambda}(N)(u - \lambda) - n_{\lambda}(N)\frac{p}{(p + (1 -p)u)(p +(1- p)\lambda)}( u - \lambda )\\
&\le N \frac{1-p}{p+(1-p)u}(\lambda - u)^2\\
&\le \frac{C_{u,p}}{N} b^2v^2.
\end{align*}
The first inequality above comes from removing the integer parts for $n_{\lambda}(N)$. The constant  $C_{u,p}$ is defined as 
\[
 C_{u,p} =  \frac{1-p}{p+(1-p)u}.
\] 
It is now straightforward to check that line \eqref{eq:xxxx} is non-negative when 
\[ b <    \frac{C_{\lambda,p}}{ 4 C_{u,p}} \,\, \text{ and }\,\,C_0 = b\frac{C_{\lambda, p}}{2}.\] 
With values of $b, C_0$
as are in the the display above, for any $c$ smaller than $b\, C_{\lambda, p} /4$, we have 
that 
\[
G^{(\lambda)}_{{\bf t}_N(\lambda)}(N)- G^{(u)}_{{\bf t}_N(\lambda)}(N) - \E[G^{(\lambda)}_{{\bf t}_N(\lambda)}(N)- G^{(u)}_{{\bf t}_N(\lambda)}(N)] \geq  c v^2 N^{-1} > 0.
\]
Using this,  we can apply Chebyshev's inequality one more time. 
In order, from Chebyshev's inequality, $(x+y)^2\leq2(x^2+y^2)$, Lemma \ref{lem:varcomp} and finally Proposition \ref{lem:exit} 
\begin{align*}
\text{Probability}&\eqref{eq:xxxx} \le \P\{| G^{(u)}_{{\bf t}_N(\lambda)}(N)- G^{(\lambda)}_{{\bf t}_N(\lambda)}(N) \\
&\phantom{xxxxxxxxxxxx}- \E[G^{(u)}_{{\bf t}_N(\lambda)}(N)- G^{(\lambda)}_{{\bf t}_N(\lambda)}(N)] |\geq  c v^2 N^{-1}\} \\
&\leq \frac{N^2}{c^2 v^4}\Var(G^{(u)}_{{\bf t}_N(\lambda)}(N)- G^{(\lambda)}_{{\bf t}_N(\lambda)}(N))\\
&\leq \frac{N^2}{c^2 v^4}\Big(\Var(G^{(u)}_{{\bf t}_N(\lambda)}(N))+\Var( G^{(\lambda)}_{{\bf t}_N(\lambda)}(N))\Big)\\
&\leq 4\frac{N^2}{c^2 v^4}\Big(\Var( G^{(\lambda)}_{{\bf t}_N(\lambda)}(N))+CN(\lambda-u)\Big)\\
&\leq 4\frac{N^2}{c^2v^4}\mathcal |A_{\cE^{(\lambda)}}|+Cb \frac{N^2}{c^2 v^3}.
\end{align*}
This together with the bound in Lemma \ref{lem:Abound} suffice for the conclusion of this proposition.
\end{proof}

\begin{proof}[Proof of Theorem \ref{thm:varub}, upper bound] 
We first bound the expected exit point for boundary with parameter $\lambda$. In what follows, $r$ is a parameter under our control, that will eventually go to $\infty$.
\begin{align*}
\E(\xi^{(\lambda)}_{e_2}(N)) &\leq rN^{2/3}+\sum_{v=rN^{2/3}}^{n_{\lambda}(N)}\P\{\xi^{(\lambda)}_{e_2}(N)>v\}\\ 
&\leq rN^{2/3} +\sum_{v=rN^{2/3}}^{\infty}C \frac{N^2}{v^3}\Big( \frac{\E(\xi^{(\lambda)}_{e_2})}{v}+1\Big)\quad\text{by}\eqref{eq:Z}\\
&\leq rN^{2/3}+\frac{C\E(\xi^{(\lambda)}_{e_2})}{r^3}+\frac{C}{r^2}N^{2/3}.
\end{align*}
Let $r$ sufficiently large so that $C/ r^3 < 1$. Then, after rearranging  the terms in the inequality above, we conclude 
\begin{align*}
\E(\xi^{(\lambda)}_{e_2}(N)) &\leq  C N^{2/3}. 
\end{align*} The variance bound follows from this, Lemma \ref{lem:Abound} and equation \eqref{eq:var3} when $m,n$ satisfy \eqref{eq:bound}.  
\end{proof}

An immediate corollary of this is the following bound in probability that is obtained directly from expression \eqref{eq:Z} is
\begin{corollary}\label{cor:CC}
Fix  $ \lambda \in (0,1)$. Then,  there exists a constant $K = K(\lambda, p)> 0$ so that for any $  r >  0 $, 
and $N$ sufficiently large
\be\label{eq:Z2}
\P\{\xi_{e_2}^{(\lambda)}(N)> r N^{2/3}\}\leq \frac{K}{r^3}.
\ee
\end{corollary}

\section{Lower bound for the variance in characteristic directions}
\label{sec:varlb}

\subsection{Down-most maximal path and Competition interface} 
In this section first we want to construct the down-most maximal path and a possible competition interface. Then we identify their properties and relations which will be crucial to find the lower bound for the order of fluctuations of the maximal path.

\subsubsection{The down-most maximal path}
Consider a triple $(I_{i,j},J_{i,j},\omega_{i,j})$ defined in \eqref{eq:bnddist}, and keep in mind the increment definition \eqref{eq:gradients}. Recall that the maximal path in the interior process collects weights only with a diagonal step with probability given by $\omega$. We define the down-most maximal path $\hat{\pi}$ starting from the target point $(m,n)$ and going backward following the rules 

\be\label{maxpathhatpi}
\hat{\pi}_{k+1}=
\begin{cases}
	\hat{\pi}_{k}+(0,1)& \text{if } G(\hat{\pi}_k+(0,1))=G(\hat{\pi}_k),\\
	\hat{\pi}_{k}+(1,0) & \textrm{if } G(\hat{\pi}_k+(1,-1))<G(\hat{\pi}_k+(0,1))\text{ and }\omega_{\hat{\pi}_k+(1,0)}=0,\\
	\hat{\pi}_{k}+(1,1) & \text{if }G(\hat{\pi}_k)=G(\hat{\pi}_k+(1,0))\text{ and }\omega_{\hat{\pi}_k+(1,1)}=1.
\end{cases}
\ee
The moment that $\hat{\pi}$ hits one of the two axes (or the origin) it starts to collect from the axis, which it has hit, down to the origin.\\
The maximal path $\hat{\pi}$ can be formalized in the following way.

The graphical representation is in \ref{fig:whole}.

\begin{figure}%
\begin{subfigure}[t]{0.33\linewidth}
\centering
\begin{tikzpicture}[>=latex]
\draw (0,0)node[xshift=0em,yshift=-0.7em]{\tiny$(i-1,j-1)$} -- (2.5,0)node[xshift=0.5em,yshift=-0.7em]{\tiny$(i,j-1)$} -- (2.5,2.5)node[xshift=0em,yshift=0.7em]{\tiny$(i,j)$} -- (0,2.5)node[xshift=-0.5em,yshift=0.7em]{\tiny$(i-1,j)$} -- (0,0);
\draw[->,red](2.3,2.3)node[xshift=-1em,yshift=-2.7em,black]{$\hat{\pi}$}--(2.3,0.1);
\draw (2.5,2.5)node[xshift=1.9em,yshift=-0.3em,black]{\tiny$\omega_{i,j}=0,1$};
\draw (2.5,2)node[xshift=1.5em,yshift=-2em,black]{\tiny$J_{i,j}=0$};
\draw (2,2.5)node[xshift=-2em,yshift=0.5em,black]{\tiny$I_{i,j}=0,1$};
\end{tikzpicture}
\end{subfigure}%
\begin{subfigure}[t]{0.33\linewidth}
\centering
\begin{tikzpicture}[>=latex]
\draw (0,0)node[xshift=0em,yshift=-0.7em]{\tiny$(i-1,j-1)$} -- (2.5,0)node[xshift=0.5em,yshift=-0.7em]{\tiny$(i,j-1)$} -- (2.5,2.5)node[xshift=0em,yshift=0.7em]{\tiny$(i,j)$} -- (0,2.5)node[xshift=-0.5em,yshift=0.7em]{\tiny$(i-1,j)$} -- (0,0);
\draw[->,red](2.3,2.3)node[xshift=-3em,yshift=-2em,black]{$\hat{\pi}$}--(0.2,0.2);
\draw (2.5,2.5)node[xshift=1.9em,yshift=-0.3em,black]{\tiny$\omega_{i,j}=1$};
\draw (2.5,2)node[xshift=1.5em,yshift=-2em,black]{\tiny$J_{i,j}=1$};
\draw (2,2.5)node[xshift=-2em,yshift=0.5em,black]{\tiny$I_{i,j}=0,1$};
\end{tikzpicture}
\end{subfigure}%
\begin{subfigure}[t]{0.33\linewidth}
\centering
\begin{tikzpicture}[>=latex]
\draw (0,0)node[xshift=0em,yshift=-0.7em]{\tiny$(i-1,j-1)$} -- (2.5,0)node[xshift=0.5em,yshift=-0.7em]{\tiny$(i,j-1)$} -- (2.5,2.5)node[xshift=0em,yshift=0.7em]{\tiny$(i,j)$} -- (0,2.5)node[xshift=-0.5em,yshift=0.7em]{\tiny$(i-1,j)$} -- (0,0);
\draw[->,red](2.3,2.3)node[xshift=-2.5em,yshift=-0.7em,black]{$\hat{\pi}$}--(0.1,2.3);
\draw (2.5,2.5)node[xshift=1.9em,yshift=-0.3em,black]{\tiny$\omega_{i,j}=0$};
\draw (2.5,2)node[xshift=1.5em,yshift=-2em,black]{\tiny$J_{i,j}=1$};
\draw (2,2.5)node[xshift=-2em,yshift=0.5em,black]{\tiny$I_{i,j}=0$};
\end{tikzpicture}%
\end{subfigure}%

\subcaptionbox{Combination of $I,J$ and\\ $\omega$ for a down ($-e_2$) step. \label{fig:foo}}[0.33\linewidth]{}%
\subcaptionbox{Combination of $I,J$ and\\ $\omega$ for a diagonal step. \label{fig:bar}}[0.33\linewidth]{}%
\subcaptionbox{Combination of $I,J$ and\\ $\omega$ for a left ($-e_1$) step.\label{fig:baz}}[0.33\linewidth]{}%
\caption{One-step backward construction for the down-most maximal path $\hat{\pi}$.  \label{fig:whole}}
\end{figure}
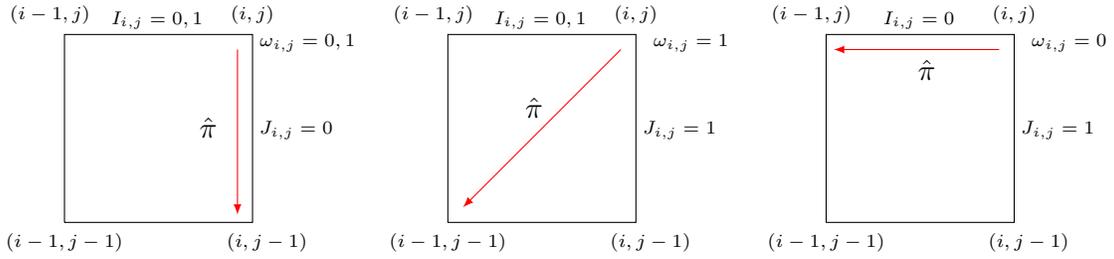

\subsubsection{The competition interface}
The competition interface is an infinite path $\varphi$ which takes only the same admissible steps as the paths we optimise over. $\varphi = \{ \varphi_0 = (0,0), \varphi_1, \ldots \}$ is completely determined by the values of $I, J$ and $\omega$. In particular, for any $k \in \N$,
\be\label{compint}
\varphi_{k+1}=
\begin{cases}
	\varphi_{k}+(0,1) & \text{if } 
		\begin{aligned} 
			&G(\varphi_k+(0,1))<G(\varphi_k+(1,0))\text{ or}\\
			&G(\varphi_k+(0,1))=G(\varphi_k+(1,0)) \text{ and }G(\varphi_k+(0,1))=G(\varphi_k),
		\end{aligned}\\ 
	\varphi_{k}+(1,0) & \textrm{if } G(\varphi_k+(1,0))<G(\varphi_k+(0,1)),\\
	\varphi_{k}+(1,1) & \text{if }G(\varphi_k+(0,1))=G(\varphi_k+(1,0)) \text{ and }G(\varphi_k+(0,1))>G(\varphi_k).
\end{cases}
\ee
In words,  the path $\varphi$ always chooses its step according to the smallest of the possible $G$-values. If they are equal, the competition interface decides to go up if the last passage time of the up and right point are equal and they are also equal to the last passage time of the starting point otherwise it takes a diagonal step. 

\begin{remark}
In literature the name \textit{competition interface} comes from the fact that it represents the threshold interface between the points which will be reached by the maximal path whose first step is right or up. Since our model is discrete, and we have three (rather than two) possible steps and our maximal path is not unique, our definition of $\varphi$ depends on our choice of maximal path; here
we chose the down-most path as our maximal path and then we accordingly defined the competition interface, so that we exploit certain good duality properties in the sequence. \qed
\end{remark} 

This being said, the partition of the plain into the two competing clusters is useful in some parts of the proofs that follow, so we would like to develop it in this setting. 
Define 
\begin{align*}
\mathcal C_{\uparrow, \nearrow} &= \{ v = (v_1, v_2) \in \Z^2_+: \text{ there exists a maximal path from $0$ to $v$} \\
&\phantom{xxxxxxxxxxxxxxxxxxxxxxxxxxxxxxxxxxxxxxxxx} \text{with first step $e_1$ or $e_1 + e_2$}\}.
\end{align*}
The remaining sites of $\Z_+^2$ are sites for which all possible maximal paths to them \emph{have to} take a horizontal first step. we denote that cluster by 
$\mathcal C_{\rightarrow} = \bZ^2_+ \setminus \mathcal C_{\uparrow, \nearrow}$. 

Some immediate observations follow. First note that the vertical axis $\{ (0, v_2) \}_{v_2 \in \N} \in \mathcal C_{\uparrow, \nearrow}$  while  $\{ (v_1, 0) \}_{v_1 \in \N} \in \mathcal C_{\rightarrow}$. We include $(0,0)\in \mathcal C_{\uparrow, \nearrow} $ in a vacuous way. 

Then observe that if 
$(v_1, v_2) \in \mathcal C_{\uparrow, \nearrow}$ then it has to be that $(v_1, y) \in \mathcal C_{\uparrow, \nearrow}$ for all $y \ge v_2$. This is a consequence of planarity. Assume that for some $y > v_2$ the maximal path $\pi _{0, (v_1, y)}$ has to take a horizontal first step. Then it will intersect with the maximal path $\pi_{0, (v_1, v_2)}$ to $(v_1, v_2)$ with a non-horizontal first step. At the point of intersection $z$, the two passage times are the same, so in fact there exists a maximal path to  $(v_1, y)$ with a non-horizontal first step: it is the concatenation of the $\pi_{0, (v_1, v_2)}$ up to site $z$ and from $z$ onwards we follow $\pi _{0, (v_1, y)}$. 

Finally, note that if $v \neq 0$ and $v \in  \mathcal C_{\uparrow, \nearrow}$ and $v + e_1 \in \mathcal C_{\rightarrow}$, it must be the case that 
\[
I _{v + e_1} = G_{0, v+e_1} - G_{0, v} = 1.
\]
Assume the contrary. Then, if the two passage times are the same, a potential maximal path to $v +e_1$ is the one that goes to $v$ without a horizontal initial step, and after $v$ it takes an $e_1$ step. This would also imply that   $v +e_1 \in  \mathcal C_{\uparrow, \nearrow}$ which is a contradiction.

These observations allow us to define a boundary between the two clusters as a piecewise linear curve $\tilde \varphi = \{ 0= \tilde \varphi_0,  \tilde \varphi_1, \ldots\}  $ which takes one of the three admissible steps, $e_1, e_2, e_1+e_2$. 
We first describe the first step of this curve when all of the $\{\om, I, J\}$ are known. (see \ref{fig:whole2}).
\be \label{eq:clusterboundary}
\tilde \varphi_1 = 
\begin{cases}
(1,0), & \text{ when } (\om_{1,1}, I_{1,0}, J_{0,1}) \in \{ (1, 0, 1), (0,0,1)\},\\
(1,1), & \text{ when } (\om_{1,1}, I_{1,0}, J_{0,1}) \in \{ (1, 0, 0), (0,0,0), (1,1,0), (1,1,1), (0,1,1) \},\\
(0,1), & \text{ when }  (\om_{1,1}, I_{1,0}, J_{0,1}) \in \{ (0,1,0)\}.
\end{cases}
\ee
\begin{figure}%
\begin{subfigure}[t]{0.33\linewidth}
\centering
\begin{tikzpicture}[>=latex]
\draw (0,0)node[xshift=0em,yshift=-0.7em]{\tiny$(0,0)$} -- (2.5,0)node[xshift=0.5em,yshift=-0.7em]{\tiny$(1,0)$} -- (2.5,2.5)node[xshift=0em,yshift=0.7em]{\tiny$(1,1)$} -- (0,2.5)node[xshift=-0.5em,yshift=0.7em]{\tiny$(1,0)$} -- (0,0);
\draw[->,red](0.2,0.2)--(2.3,2.3)node[xshift=-3.3em,yshift=-2em,black]{$\tilde{\varphi}_1$};
\draw (2.5,2.5)node[xshift=1.9em,yshift=-0.3em,black]{\tiny$\omega_{1,1}=0,1$};
\draw (1.3,-0.3)node{\tiny$I_{1,0}=0$};
\draw (-0.6,1.3)node{\tiny$J_{0,1}=0$};
\end{tikzpicture}
\end{subfigure}%
\begin{subfigure}[t]{0.33\linewidth}
\centering
\begin{tikzpicture}[>=latex]
\draw (0,0)node[xshift=0em,yshift=-0.7em]{\tiny$(0,0)$} -- (2.5,0)node[xshift=0.5em,yshift=-0.7em]{\tiny$(1,0)$} -- (2.5,2.5)node[xshift=0em,yshift=0.7em]{\tiny$(1,1)$} -- (0,2.5)node[xshift=-0.5em,yshift=0.7em]{\tiny$(1,0)$} -- (0,0);
\draw[->,red](0.2,0.2)--(2.3,2.3)node[xshift=-3.3em,yshift=-2em,black]{$\tilde{\varphi}_1$};
\draw (2.5,2.5)node[xshift=1.6em,yshift=-0.3em,black]{\tiny$\omega_{1,1}=1$};
\draw (1.3,-0.3)node{\tiny$I_{1,0}=1$};
\draw (-0.7,1.3)node{\tiny$J_{0,1}=0,1$};
\end{tikzpicture}
\end{subfigure}%
\begin{subfigure}[t]{0.33\linewidth}
\centering
\begin{tikzpicture}[>=latex]
\draw (0,0)node[xshift=0em,yshift=-0.7em]{\tiny$(0,0)$} -- (2.5,0)node[xshift=0.5em,yshift=-0.7em]{\tiny$(1,0)$} -- (2.5,2.5)node[xshift=0em,yshift=0.7em]{\tiny$(1,1)$} -- (0,2.5)node[xshift=-0.5em,yshift=0.7em]{\tiny$(1,0)$} -- (0,0);
\draw[->,red](0.2,0.2)--(2.3,2.3)node[xshift=-3.3em,yshift=-2em,black]{$\tilde{\varphi}_1$};
\draw (2.5,2.5)node[xshift=1.6em,yshift=-0.3em,black]{\tiny$\omega_{1,1}=0$};
\draw (1.3,-0.3)node{\tiny$I_{1,0}=1$};
\draw (-0.6,1.3)node{\tiny$J_{0,1}=1$};
\end{tikzpicture}%
\end{subfigure}
\subcaptionbox{Combination of $I,J$ and\\ $\omega$ for a diagonal step. \label{fig:foo}}[0.33\linewidth]{}%
\subcaptionbox{Combination of $I,J$ and\\ $\omega$ for a diagonal step. \label{fig:bar}}[0.33\linewidth]{}%
\subcaptionbox{Combination of $I,J$ and\\ $\omega$ for a diagonal step.\label{fig:baz}}[0.33\linewidth]{}%

\begin{subfigure}[t]{0.4\textwidth}
\centering
\begin{tikzpicture}[>=latex]
\draw (0,0)node[xshift=0em,yshift=-0.7em]{\tiny$(0,0)$} -- (2.5,0)node[xshift=0.5em,yshift=-0.7em]{\tiny$(1,0)$} -- (2.5,2.5)node[xshift=0em,yshift=0.7em]{\tiny$(1,1)$} -- (0,2.5)node[xshift=-0.5em,yshift=0.7em]{\tiny$(1,0)$} -- (0,0);
\draw[->,red](0.15,0.15)node[yshift=3em,xshift=0.9em,black]{$\tilde{\varphi}_1$}--(0.15,2.35);
\draw (2.5,2.5)node[xshift=1.9em,yshift=-0.3em,black]{\tiny$\omega_{1,1}=0,1$};
\draw (1.3,-0.3)node{\tiny$I_{1,0}=1$};
\draw (-0.6,1.3)node{\tiny$J_{0,1}=0$};
\end{tikzpicture}
\end{subfigure}%
\begin{subfigure}[t]{0.4\textwidth}
\centering
\begin{tikzpicture}[>=latex]
\draw (0,0)node[xshift=0em,yshift=-0.7em]{\tiny$(0,0)$} -- (2.5,0)node[xshift=0.5em,yshift=-0.7em]{\tiny$(1,0)$} -- (2.5,2.5)node[xshift=0em,yshift=0.7em]{\tiny$(1,1)$} -- (0,2.5)node[xshift=-0.5em,yshift=0.7em]{\tiny$(1,0)$} -- (0,0);
\draw[->,red](0.15,0.15)node[xshift=3em,yshift=0.7em,black]{$\tilde{\varphi}_1$}--(2.35,0.15);
\draw (2.5,2.5)node[xshift=1.6em,yshift=-0.3em,black]{\tiny$\omega_{1,1}=0$};
\draw (1.3,-0.3)node{\tiny$I_{1,0}=0$};
\draw (-0.6,1.3)node{\tiny$J_{0,1}=1$};
\end{tikzpicture}%
\end{subfigure}
\subcaptionbox{Combination of $I,J$ and\\ $\omega$ for an up step. \label{fig:bar}}[0.4\linewidth]{}%
\subcaptionbox{Combination of $I,J$ and\\ $\omega$ for a right step.\label{fig:baz}}[0.4\linewidth]{}%
\caption{Constructive admissible steps for $\tilde{\varphi}_1$.  \label{fig:whole2}}
\end{figure}

From this definition we see that $\tilde\varphi_1$ stays on the $x$-axis only when $I_{1,0} = 0$ and $J_{0,1} = 1$. If that is the case, repeat the steps in \eqref{eq:clusterboundary} until $\tilde\varphi$ increases its $y$-coordinate and changes level. Any time $\tilde\varphi$ changes level from $\ell-1$ to $\ell $,
 it takes horizontal steps (the number of steps could be 0) until a site $(v_\ell, \ell)$ where $(v_{\ell}, \ell) \in \mathcal C_{\uparrow, \nearrow}$ but $(v_{\ell}+1, \ell) \in \mathcal C_{\rightarrow}$. In that case, $I_{v_{\ell}+1, \ell} = 1$, by the second and third observations above, and $\tilde\varphi$ will change level, again following the steps in \eqref{eq:clusterboundary}.

From the description of the evolution of $\tilde \varphi$, starting from \eqref{eq:clusterboundary} and evolving as we describe in the previous paragraph, the definition of the competition interface $\varphi$ in \eqref{compint}, implies as piecewise linear curves, 
\be \label{eq:phiordering}
\varphi \ge \tilde \varphi,
\ee
i.e.\ if $(x, y_1) \in \varphi$ and $(x, y_2) \in \tilde \varphi$ then, $y_1 \ge y_2$. Similarly, if $(x_1, y) \in \varphi$ and $(x_2, y) \in \tilde \varphi$ then, $x_1 \le x_2$.
Moreover, if $u\in\Z^2_+\not\in\tilde\varphi$ then it has to belong to one of the clusters; $\mathcal C_{\rightarrow}$ if $u$ is below $\tilde\varphi$ and $\mathcal C_{\uparrow, \nearrow}$ otherwise. (see \ref{fig:comphi}).

\begin{figure}[h]
\centering
\begin{tikzpicture}[scale=0.7, >=latex]
\draw[fill=sussexp!30,sussexp!30] (0,0)--(1,1)--(1,2)--(2,3)--(3,3)--(3,4)--(4,5)--(6,5)--(7,6)--(7,7)--(9,7)--(9,9)--(9,0)--(0,0);
\draw[fill=sussexg!30,sussexg!30] (0,0)--(1,1)--(1,2)--(2,3)--(3,3)--(3,4)--(4,5)--(6,5)--(7,6)--(7,7)--(9,7)--(9,9)--(0,9)--(0,0);
\draw[->] (-0.2,0) -- (9.5,0) ;
  \draw[->] (0,-0.2) -- (0,9.5) ;
\draw[blue, line width=2pt] (0,0)--(1,1)--(1,2)--(2,3)--(3,3)--(3,4)--(4,5)--(6,5)--(7,6)--(7,7)--(9,7)--(9,9)node[xshift=1.5em,black]{\tiny$(m,n)$};
\draw[nicos-red, line width = 2pt] (0,0)--(0,1)--(0,2)--(1,3)--(1,4)--(3,4)--(4,5)--(6,5)--(6,8)--(8,8)--(9,9);
\matrix [draw] at (12,8.5) {
\draw [blue, line width=2pt] (0,0)--(0.3,0)node[xshift=0.7em,black]{\tiny$\tilde{\varphi}$}; \\
\draw [nicos-red, line width=2pt] (0,0)--(0.3,0)node[xshift=0.7em,black]{\tiny$\varphi$};\\
\draw [fill=sussexp!30,sussexp!30] (0,0) rectangle (0.3,0.1)node[yshift=-0.2em, xshift=1em,black]{\tiny$\mathcal C_{\rightarrow}$};\\
\draw [fill=sussexg!30,sussexg!30]  (0,0) rectangle (0.3,0.1)node[yshift=-0.2em, xshift=1.3em,black]{\tiny$\mathcal C_{\uparrow, \nearrow}$};\\
};
\end{tikzpicture}
\caption{Graphical representation of $\tilde{\varphi}$ and $\varphi$. Both curves can be thought as competition interfaces. $\tilde{\varphi}$ separates competing clusters, depending on the first step of the right-most maximal path, while $\varphi$ follows the smallest increment of passage times with a rule to break ties. As curves they are geometrically ordered, $\tilde \varphi \le \varphi$.}
\label{fig:comphi}
\end{figure}
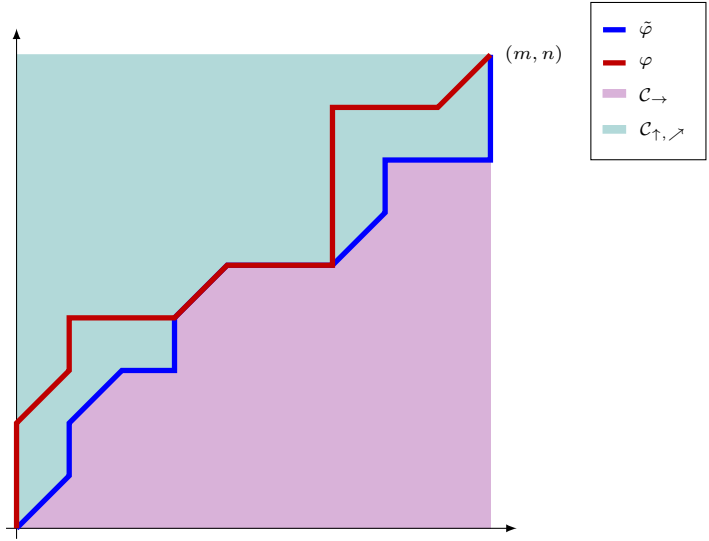
\subsubsection{The reversed process}
Let $(m,n)$ with $m,n>0$ be the target point. Define
\be\label{revproc}
G^*_{i,j} = G_{m,n} - G_{m-i, n -j},\qquad\text{for }0\leq i<m\text{ and }0\leq i<n.
\ee

It represents the time to reach point $(i,j)$ starting from $(m,n)$ for the reversed process.
We also define the new edge and the bulk weights by 
\begin{align}
I^*_{i,j} &= I_{m-i+1, n-j}, \quad \text{ when } i \ge 1, j \ge 0\label{eq:istar} \\
J^*_{i,j} &= J_{m-i, n-j +1}, \quad \text{ when } i \ge 0, j \ge 1\\
\omega^*_{i,j} &= \alpha_{m-i, n-j}, \quad \text{ when } i \ge 1, j \ge 1.\label{intrev}
\end{align}
Then we have the reverse identities. 
\begin{lemma}\label{reversal}
Let $I^*$ and $J^*$ be respectively the horizontal and vertical increment for the reversed process. Then, for $0\leq i<m\text{ and }0\leq i<n$, we have
\begin{align}
I^*_{i,j} &= \omega^*_{i,j} \vee I^*_{i,j-1} \vee J^*_{i-1,j} - J^*_{i-1,j} = G^*_{i,j} - G^*_{i-1, j} \\
J^*_{i,j} &= \omega^*_{i,j} \vee I^*_{i,j-1} \vee J^*_{i-1,j} - I^*_{i,j-1} = G^*_{i,j} - G^*_{i, j-1}.
\end{align}
\end{lemma}
\begin{proof} 
First note that  
\begin{align*}
 I_{m-i+1, n-j} &= G_{m-i+1,n-j} - G_{m-i, n-j} \\
 &= G_{m-i+1,n-j} - G_{m,n} + G_{m,n} - G_{m-i, n-j} = G^*_{i,j} - G^*_{i-1,j}.
\end{align*}
by \eqref{revproc}.
We also prove the other identity only for the $I^*_{i,j}$ and leave the proof for the second set of equations to the reader.
A direct substitution to the right-hand side gives 
\begin{align*}
 \omega^*_{i,j} &\vee I^*_{i,j-1} \vee J^*_{i-1,j} - J^*_{i-1,j} \\
 &= \alpha_{m-i, n-j}\vee I_{m-i+1, n-j+1} \vee J_{m-i+1, n-j +1} - J_{m-i+1, n-j +1}\\
 &= (\alpha_{m-i, n-j} - J_{m-i+1, n-j +1})\vee (G_{m-i+1,n-j} - G_{m-i,n-j+1}) \vee0\\
 &= (\alpha_{m-i, n-j} - J_{m-i+1, n-j +1})\vee (G_{m-i+1,n-j} \pm G_{m-i, n-j} - G_{m-i,n-j+1}) \vee0\\
 &= (\alpha_{m-i, n-j} - (\omega_{m-i+1, n-j +1}\vee I_{m-i+1, n-j}\vee J_{m-i, n-j +1} - I_{m-i+1, n-j}) )\\
 &\phantom{xxxxxx}\vee(I_{m-i+1, n-j} - J_{m-i,n -j +1})\vee0 \\
 &=I_{m-i+1, n-j} + \Big(  (\alpha_{m-i, n-j} - \omega_{m-i+1, n-j +1}\vee I_{m-i+1, n-j}\vee J_{m-i, n-j +1})\\
 &\phantom{xxxxxxxxxxxxxxxxxxxxxxxxxxxxxxx}\vee(- J_{m-i,n -j +1})\vee (-I_{m-i+1, n-j}) \Big).
\end{align*} 
Focus on the expression in the parenthesis. We will show that it is always 0, and therefore the lemma follows by \eqref{eq:istar}.
We use equations \eqref{eq:4} and \eqref{eq:5}. If $(I_{m-i+1, n-j}, J_{m-i, n -j +1})$ $= (1,1)$ then $\alpha_{m-i, n-j} =1$ and the first maximum is zero. Similarly, when  the triple $(\omega_{m-i+1, n-j +1}, I_{m-i+1, n-j}, J_{m-i, n-j +1}) = (0,0,0)$, $\alpha_{m-i, n-j} =0$ and the value is zero again. When exactly one of $I_{m-i+1, n-j}, J_{m-i, n -j +1}$ is zero the overall maximum in the parenthesis is 0, irrespective of the values of $\alpha_{m-i, n-j}, \omega_{m-i+1, n-j +1}$. Finally, when $\omega_{m-i+1, n-j +1} =1$ and both the increment variables $(I_{m-i+1, n-j}, J_{m-i, n -j +1}) = (0,0)$, the first term is either $0$ or $-1$ and again the overall maximum is zero. 
\end{proof}
Throughout the paper quantities defined in the reversed process will be denoted by a superscript $^*$, and they will always be equal in distribution to their original forward versions.
\subsubsection{Competition interface for the forward process vs maximal path for the  reversed process}
We want to show that the competition interface defined in \eqref{compint} is always below or coincides  (as piecewise linear curves) with the down - most maximal path $\hat{\pi}^*$ for the reversed process. The steps of the competition interface for the forward process coincide with those of $\hat \pi^*$ in all cases, except when $(I_{i,j},J_{i,j},\omega_{i,j}) = (0,1,1)$. In that case, $\hat \pi^*$ will go diagonally up, while $\phi$ will move horizontally. Thus, $\varphi$ is to the right and below $\hat \pi ^*$ as curves.

Now, define
\be \label{eq:phiproj}
\begin{aligned}
v(n)&=\inf\{i:(i,n)=\varphi_k\text{ for some }k\geq0\}\\
w(m)&=\inf\{j:(m,j)=\varphi_k\text{ for some }k\geq0\}
\end{aligned}
\ee
with the convention $\inf\varnothing=\infty$. In words, the point $(v(n),n)$ is the left-most point of the competition interface on the horizontal line $j=n$, while $(m,w(m))$ is the lowest point on the vertical line $i=m$. This observation implies 
\be\label{vwrel}
 v(n)\geq m\implies w(m)< n\qquad\text{or } \quad w(m)\geq n\implies v(n)> m.
\ee
Then, on the event $\{  v(n) \ge m \}$, we know that $\hat \pi^*$ will hit the north boundary of the rectangle at a site $( \ell , n)$ so that 
\[
m - \ell = \xi^*_{e_1}({\hat\pi^*}),\quad  \ell \le v(n).
\]
Then, we have just showed that 

\begin{lemma}\label{revintcon}
Let $\varphi$ be the competition interface constructed for the process $G^{(\lambda)}$ and $\hat \pi^*$ the down-most maximal path for the process $G^{*,(\lambda)}$ defined by \eqref{revproc} from $(m,n)$ to $(0,0)$. Then on the event $\{  v(n) \ge m \}$,
\be\label{xire}
m - v(n) \le \xi^{*(\lambda)}_{e_1}({\hat\pi^*})
\ee
\end{lemma}
Finally, note that by reversed process definition we have 
\be\label{eqdis}
\xi_{e_1}^{*(\lambda)}\dis\xi_{e_1}^{(\lambda)}.
\ee
\subsection{Last passage time under different boundary conditions}
In our setting the competition interface is important because it bounds the region where the boundary conditions on the axes are felt. For this reason we want to give a Lemma which describes how changes in the boundary conditions are felt by the increments in the interior part.
\begin{lemma}\label{lem:2w}
Given two different weights $\{\om_{i,j}\}$ and $\{\tilde{\om}_{i,j}\}$ which satisfy $\om_{0,0}=\tilde{\om}_{0,0}$, $\om_{0,j}\geq\tilde{\om}_{0,j}$, $\om_{i,0}\leq\tilde{\om}_{i,0}$ and $\om_{i,j}=\tilde{\om}_{i,j}$ for all $i,j\geq1$. Then all increments satisfy $I_{i,j}\leq\tilde{I}_{i,j}$ and $J_{i,j}\geq \tilde{J}_{i,j}$.
\end{lemma}
\begin{proof}
By following the same corner-flipping inductive proof as that of Lemma \ref{lem:drp-bus} one can show that the statement holds for all increments between points in $L_\psi  \cup \mathcal I_\psi$ where $L_\psi$ and  $\mathcal I_\psi$ are respectively defined in \eqref{eq:patheses} and \eqref{eq:interior} for those paths for which $\mathcal I_\psi$  is finite. The base case is when $\mathcal I_\psi$ is empty and the statement follows from the assumption made on the weights $\{\om_{i,j}\}$ and $\{\tilde{\om}_{i,j}\}$ and from the definition of the increments made in \eqref{eq:4}.
\end{proof}

\begin{lemma}\label{bosw}
We are in the settings of Lemma \ref{lem:2w}. Let $G^{\mathcal W=0}$ (resp.$G^{\mathcal S=0}$) be the last passage times of a system where we set $\tilde{\om}_{0,j}=0$ for all $j\geq1$ (resp. $\om_{i,0}=0$) and the paths are allowed to collect weights while on the boundaries. Let $v(n)$ be given by \eqref{eq:phiproj}. 

Then, for $v(n) <  m_1 \le m_2$,
\be\label{comw}
\begin{aligned}
G_{(1,1),(m_2,n)}-G_{(1,1),(m_1,n)}&\leq G^{\mathcal W=0}_{(0,0),(m_2,n)}-G^{\mathcal W=0}_{(0,0),(m_1,n)}\\
&=G_{(0,0),(m_2,n)}-G_{(0,0),(m_1,n)}.
\end{aligned}
\ee
Alternatively, for $0\leq m_1 \le m_2 < v(n)$,
\be\label{coms}
\begin{aligned}
G_{(1,1),(m_2,n)}-G_{(1,1),(m_1,n)}&\geq G^{\mathcal S=0}_{(0,0),(m_2,n)}-G^{\mathcal S=0}_{(0,0),(m_1,n)}\\
&=G_{(0,0),(m_2,n)}-G_{(0,0),(m_1,n)}.
\end{aligned}
\ee
\end{lemma}
\begin{proof}
We prove \eqref{coms} and similar arguments prove \eqref{comw}.
The first inequality in \eqref{coms} follows from Lemma \ref{lem:2w} in the case $ \tilde \om_{0,j}= \tilde \om_{i,0}=0$. The subsequent equality comes from the fact that if $v(n)\ge m_1 \ge m_2$. By \eqref{eq:phiproj} the target points $(m_1,n)$ and $(m_2,n)$ are above the competition interface $\varphi$ and therefore, by \eqref{eq:phiordering} are strictly above $\tilde\varphi$. This implies that $(m_1, n)$ and $(m_2, n)$ belong to the cluster $\mathcal C_{\uparrow, \nearrow}$ and therefore we can choose the respective maximal paths to not take a horizontal first step. In turn, the maximal path does not need to go through the $x$-axis and hence it does not see the boundary values $\omega_{i, 0}$. Thus, $G^{\mathcal S=0}_{(0,0),(m,n)}=G_{(0,0),(m_,n)}$.
\end{proof}

\subsection{Lower bound}
In this section we prove the lower bound for the order of the variance. Before giving the proof we need to prove two preliminary lemmas. For the rest of this section, whenever we say maximal path, we mean the down-most maximal path.
\begin{lemma}\label{lem:probban}
Let $a,b>0$ two positive numbers. Then there exist a positive integer $N_0=N(a,b)$ and constant $C=C(a,b)$ such that for all $N>N_0$ we have
\begin{align}\notag
\mathbb{P}\Big\{\sup_{0\leq z\leq aN^{2/3}}\{\sS^{(u)}_z+G_{(z\vee1,1),(m_u(N),n_u(N))}&-G_{(1,1),(m_u(N),n_u(N))}\}\geq bN^{1/3}\Big\}\\
&\phantom{xxxxxxxxxxxxxxx}\leq Ca^3(b^{-3}+b^{-6}). \label{lempstate}
\end{align}
\end{lemma}
\begin{proof}
First note that if the supremum in the probability is attained at $z = 0$ then the expression in the braces is tautologically $0$ and the statement of the lemma is vacuously true. Therefore without loss of generality, we can prove the bound for the supremum when $1 \le z \le aN^{2/3}$.
 
Select and fix any parameter $0< r < b/a$ and let $N$ large enough. The exact dependence or $r$ on the parameters $a$ and $b$ will be obtained later in the proof. 
Define $\lambda$ by 
\be\label{lam}
\lambda=u-rN^{-1/3}.
\ee
and use it to define boundary weights on both axes using that parameter and independently of the original boundary weights with parameter $u$. 
The environment in the bulk is the same for both processes.  
Let  $\varphi^{(\lambda)}$ be the competition interface under environment $\omega^{(\lambda)}$ and let $v^{(\lambda)}$ be as in equation \eqref{eq:phiproj}. 
Restrict on the event $v^\lambda(n) > m$. Define the increment $\mathcal V^{(\lambda)}_{z-1}=G^{(\lambda)}_{(0,0),(m,n)}-G^{(\lambda)}_{(0,0), (m-z+1,n)}$. Then use Lemma \ref{bosw} to obtain
\[
G_{(1,1),(m,n)}-G_{(1,1),(m-z+1,n)}\geq \mathcal V^{(\lambda)}_{z-1}.
\]
Recall that $\mathcal  V^{(\lambda)}_{z-1}$ is a sum of i.i.d. Bernoulli$(\lambda)$ variables and it is independent of $\mathscr S^{(u)}_z$. 
When $(m,n)$ equals the characteristic direction  $(m_u(N), n_u(N))$ corresponding to $u$, 
\begin{align}
\mathbb{P}\big\{\sup_{1\leq z\leq aN^{2/3}}\{\sS^{(u)}_z& +G_{(z,1),(m_u(N),n_u(N))}-G_{(1,1),(m_u(N),n_u(N))}\}\geq bN^{1/3}\big\} \notag \\
&\leq\mathbb{P}\Big\{v^{(\lambda)}\Big(\Big\lfloor \frac{N}{p}\big(p + (1-p)u\big)^2\Big\rfloor\Big)\le \lfloor N\rfloor\Big\} \notag \\
&\phantom{xxxxxxxxxxxxxx} +\mathbb{P}\Big\{\sup_{1\leq z\leq aN^{2/3}}\{\sS^{(u)}_z-\mathcal V^{(\lambda)}_{z-1}\}\geq bN^{1/3}\Big\} \notag\\
&\leq\mathbb{P}\Big\{v^{(\lambda)}\Big(\Big\lfloor \frac{N}{p}\big(p + (1-p)u\big)^2\Big\rfloor\Big)\le \lfloor N\rfloor\Big\} \label{plb} \\
&\phantom{xxxxxxxxxxxxxx} +\mathbb{P}\Big\{\sup_{1\leq z\leq aN^{2/3}}\{\sS^{(u)}_{z-1}-\mathcal V^{(\lambda)}_{z-1}\}\geq bN^{1/3}- 1\Big\}. \label{plb1}
\end{align}
We bound the two probabilities separately. We begin with \eqref{plb1}.
Define the martingale as $M_{z-1}=\sS^{(u)}_{z-1}-\mathcal V^{(\lambda)}_{z-1}-\E[\sS^{(u)}_{z-1}-\mathcal V^{(\lambda)}_{z-1}]$, and note that for $1 \le z\leq aN^{2/3}$,
\be\label{expmar}
\E[\sS^{(u)}_{z-1}-\mathcal V^{(\lambda)}_{z-1}]=(z-1)u-(z-1)\lambda\leq raN^{1/3}.
\ee
From \eqref{expmar} follows that 
\[\sS^{(u)}_{z-1}-\mathcal V^{(\lambda)}_{z-1}\leq M_{z-1}+raN^{1/3}.\]
Using this result and taking $N$ large enough so that 
\begin{equation}\label{eq:condonb}
b>ra +N^{-1/3}
\end{equation}
we get by Doob's inequality, for any $d\geq1$.
\begin{align}
\mathbb{P}\big\{\sup_{1\leq z\leq aN^{2/3}}&\{\sS^{(u)}_{z-1}-\mathcal V^{(\lambda)}_{z-1}\} \geq bN^{1/3} -1\big\}\notag\\
&\leq\mathbb{P}\big\{\sup_{1\leq z\leq aN^{2/3}}M_{z-1}\geq N^{1/3}(b-ra-N^{-1/3})\big\}\nonumber\\
&\leq\frac{C(d)N^{-d/3}}{(b-ra-N^{-1/3})^p}\E[|M_{\lfloor aN^{2/3}\rfloor}|^d]\leq\frac{C(d,u)a^{d/2}}{(b-ra-N^{-1/3})^d}.\label{eq:Doob}
\end{align}
Then for $N\geq 4^3b^{-3}$ the above bound is further dominated by 
$
C(d,u)a^{d/2}{(\frac{3b}{4}-ra\big)^{-d}}
$
which becomes $C(d,u)a^3b^{-6}$ once we choose
\begin{equation}\label{eq:r}
r=\frac{b}{4a},
\end{equation}
$d=6$, and properly re-define the  constant $C(d,u)$. This concludes the bound for \eqref{plb1}.

For \eqref{plb}, we rescale $N$ as 
\[
N'=\Big(\frac{p+(1-p)u}{p+(1-p)\lambda}\Big)^2N.
\]
Then we write
\[
\mathbb{P}\Big\{v^{(\lambda)}\Big(\Big\lfloor  \frac{N'}{p}\big(p + (1-p)\lambda\big)^2\Big\rfloor\Big)<\Big\lfloor \Big(\frac{p+(1-p)\lambda}{p+(1-p)u}\Big)^2N'\Big\rfloor\Big\}
\]
Since $u>\lambda$, then
\[
\Big\lfloor \Big(\frac{p+(1-p)\lambda}{p+(1-p)u}\Big)^2N'\Big\rfloor\leq\lfloor N' \rfloor .
\]
Thus, by redefining \eqref{eq:bound} and \eqref{xire} with $N'$ and $\lambda$, we have that the event $v^{(\lambda)}(\lfloor  \frac{N'}{p}(p + (1-p)\lambda)^2\rfloor)<\lfloor (\frac{p+(1-p)\lambda}{p+(1-p)u})^2N'\rfloor$ is equivalent to
\begin{align*}
\xi_{e_1}^{*(\lambda)}(N')&\geq\lfloor N' \rfloor-v^{(\lambda)}\Big(\Big\lfloor  \frac{N'}{p}\big(p + (1-p)\lambda\big)^2\Big\rfloor\Big)\\
&>\lfloor N' \rfloor-\Big\lfloor \Big(\frac{p+(1-p)\lambda}{p+(1-p)u}\Big)^2N'\Big\rfloor.
\end{align*}
By \eqref{eqdis}, we conclude
\begin{align}
\mathbb{P}\Big\{v^{(\lambda)}\Big(\Big\lfloor \frac{N}{p}\big(p + (1-p)u\big)^2\Big\rfloor\Big)&<\lfloor N\rfloor\Big\}\nonumber\\
&=\mathbb{P}\Big\{\xi_{e_1}^{(\lambda)}(N')>\lfloor N' \rfloor-\Big\lfloor \Big(\frac{p+(1-p)\lambda}{p+(1-p)u}\Big)^2N'\Big\rfloor\Big\}.
\end{align}
Utilizing the definitions \eqref{lam} and \eqref{eq:r} of $\lambda$ and $r$, for $N\geq N_0$ there exists a constant $C=C(u)$ such that 
\[\lfloor N' \rfloor-\Big\lfloor \Big(\frac{p+(1-p)\lambda}{p+(1-p)u}\Big)^2N'\Big\rfloor\geq CrN'^{2/3}.\]
Combining this with Corollary \ref{cor:CC} and definition \eqref{eq:r} of $r$ we get the bound 
\begin{align}
\mathbb{P}\Big\{v^{(\lambda)}\Big(\Big\lfloor \frac{N}{p}\big(p + (1-p)u\big)^2\Big\rfloor\Big)<\lfloor N\rfloor\Big\}&\leq\mathbb{P}[\xi_{e_1}^{(\lambda)}(N')>CrN'^{2/3}]\nonumber\\
&\leq Cr^{-3}\leq C(a/b)^3.\label{eq:vresult}
\end{align}
The result now follows.
\end{proof}
 The other Lemma gives an asymptotic limit of the probability order of the exit point from the $x$-axis. We will discuss the exit point from the $y$-axis as a Corollary of this Lemma.
\begin{lemma}\label{lem:asym} Let $u \in (0,1)$ and $(m_u(N), n_u(N))$ the characteristic direction. Then the exit point of a maximal path from $0$ to  $(m_u(N), n_u(N))$ can collect, satisfies  
\[
\lim_{\delta\to0}\varlimsup_{N\to\infty}\mathbb{P}\Big\{0\leq {\xi^{(u)}_{e_1}(N)} \vee {\xi^{(u)}_{e_2}(N)}  \leq\delta N^{2/3}\Big\}=0.
\]
\end{lemma}
 \begin{proof} We only show the result for ${\xi^{(u)}_{e_1}(N)}$. The same result for  ${\xi^{(u)}_{e_2}(N)}$ follows by interchanging vertical and horizontal directions and the fact that both boundaries have Bernoulli variables. 
  
First pick a parameter $\delta > 0$. 
Recall that $\xi^{(u)}_{e_1}(N)=0$ if the down-most maximal path makes the first step diagonally or up. Also keep in mind that $\xi^{(u)}_{e_1}(N)=0$ is the right-most possible exit point, therefore all paths that exit later, have to have a obtain smaller passage time.
Then, we may bound 
\begin{align*}
\mathbb{P}\{0\leq\xi^{(u)}_{e_1}(N)\leq\delta N^{2/3}\}&\leq\mathbb{P}\big\{\sup_{\delta N^{2/3}<x\leq N^{2/3}}\{\sS^{(u)}_{x}+G_{(x,1),(m_u(N),n_u(N))}\}\\
&\hspace{3cm}<\sup_{0\leq x\leq \delta N^{2/3}}\{\sS^{(u)}_{x}+G_{(x\vee1,1),(m_u(N),n_u(N))}\}\big\}.
\end{align*}
Then, we subtract the term $G_{(1,1),(m(N),n(N))}$ from both sides and we bound the resulting probability from above by
\begin{align}
&\mathbb{P}\big\{\sup_{\delta N^{2/3}<x\leq N^{2/3}}\{\sS^{(u)}_{x}+G_{(x,1),m_u(N),n_u(N))}-G_{(1,1),(m_u(N),n_u(N))}\}\nonumber\\
&\hspace{4cm}<\sup_{0\leq x\leq \delta N^{2/3}}\{\sS^{(u)}_{x}+G_{(x\vee1,1),(m_u(N),n_u(N))}-G_{(1,1),(m_u(N),n_u(N))}\}\big\}\nonumber\\
&\hspace{1.5cm}\leq\mathbb{P}\big\{\sup_{\delta N^{2/3}<x\leq N^{2/3}}\{\sS^{(u)}_{x}+G_{(x,1),(m_u(N),n_u(N))}-G_{(1,1),(m_u(N),n_u(N))}\}<bN^{1/3}\big\}\label{fte}\\
&\hspace{1.5cm}+\mathbb{P}\big\{\sup_{0\leq x\leq \delta N^{2/3}}\{\sS^{(u)}_{x}+G_{(x\vee1,1),(m_u(N),n_u(N))}-G_{(1,1),(m_u(N),n_u(N))}\}>bN^{1/3}\big\}.\label{ste}
\end{align}
\eqref{ste} is bounded from above using Lemma \ref{lem:probban} by $C\delta^3(b^{-3}+b^{-6})$. 

To bound \eqref{fte} we use similar arguments that we employed in the proof of Lemma \ref{lem:probban}. Define an auxiliary parameter $\lambda$  
\be \label{lamu}
\lambda=u+rN^{-1/3},
\ee
where conditions on $r$ will be specified in the course of the proof.
From Lemma \ref{bosw} the following inequalities hold
\begin{align*}
G_{(x,1),(m_u(N),n_u(N))}-G_{(1,1),(m_u(N),n_u(N))}&\geq G^{(\lambda)}_{(0,0),(m_u(N),n_u(N))}-G^{(\lambda)}_{(0,0), (m_u(N)-x+1,n_u(N))}\\
&=-\mathcal V^{(\lambda)}_{x-1}\geq-\mathcal V^{(\lambda)}_{x}.
\end{align*}
whenever $v^{(\lambda)}\Big(\Big\lfloor \frac{N}{p}\big(p + (1-p)u\big)^2\Big\rfloor\Big)\leq\lfloor N\rfloor-x.$
Using these, we have
\begin{align}
&\mathbb{P}\big\{\sup_{\delta N^{2/3}<x\leq N^{2/3}}\{\sS^{(u)}_{x}+G_{(x,1),(m_u(N),n_u(N))}-G_{(1,1),(m_u(N),n_u(N))}\}<bN^{1/3}\big\}\nonumber\\
&\hspace{2cm}\leq\mathbb{P}\big\{v^{(\lambda)}\Big(\Big\lfloor \frac{N}{p}\big(p + (1-p)u\big)^2\Big\rfloor\Big)>\lfloor N\rfloor-N^{2/3}\big\}\label{2fte}\\
&\hspace{5cm}+\mathbb{P}\big\{\sup_{\delta N^{2/3}<x\leq N^{2/3}}\{\sS^{(u)}_{x}-\mathcal V^{(\lambda)}_{x}\}<bN^{1/3}\big\}.\label{2ste}
\end{align}
We claim that, for $\eta>0$ and parameter $r$, it is possible to fix $\delta,b>0$ small enough so that, for some $N_0<\infty$, the probability in \eqref{2ste} satisfies
\begin{equation}\label{eq:p5}
\mathbb{P}\big\{\sup_{\delta N^{2/3}<x\leq N^{2/3}}\{\sS^{(u)}_{x}-\mathcal V^{(\lambda)}_{x}\}<bN^{1/3}\big\}\leq\eta\qquad\text{for all }N\geq N_0.
\end{equation}

In order to prove this, we use a scaling argument: Uniformly over $y\in[\delta,1]$ as $N\to\infty$,
\begin{align*}
N^{-1/3}\E[\sS^{(u)}_{\lfloor yN^{2/3}\rfloor}-\mathcal V^{(\lambda)}_{\lfloor yN^{2/3}\rfloor}]=N^{-1/3}(\lfloor yN^{2/3}\rfloor u-\lfloor yN^{2/3}\rfloor(u+rN^{-1/3}))\to -ry
\end{align*}
and
\begin{align*}
N^{-2/3}\Var(\sS^{(u)}_{\lfloor yN^{2/3}\rfloor}-\mathcal V^{(\lambda)}_{\lfloor yN^{2/3}\rfloor})&= y\bigg(\sqrt{u(1-u)}+\sqrt{(u+rN^{-1/3})(1-u-rN^{-1/3})}\bigg)^2\\
&\to4u(1-u)y=\sigma^2(u)y 
\end{align*}
Since we are scaling the supremum of a random walk with bounded increments, the probability \eqref{eq:p5} converges as $N\to\infty$, to
\[
\mathbb{P}\big\{\sup_{\delta\leq y\leq1}\{\sigma(u)\mathfrak{B}(y)-ry\}\leq b\big\}
\]
where $\mathfrak{B}(\cdot)$ is a standard Brownian motion. The random variable
\[\sup_{\delta \leq y\leq1}\{\sigma(u)\mathfrak{B}(y)-ry\}\]
is positive almost surely when $\delta$ is sufficiently small. Therefore, the above probability is less than $\eta/2$ for a suitably small $b$. This implies \eqref{eq:p5}.

Finally we bound \eqref{2fte}. Using \eqref{vwrel} and the transpose environment $\tilde \om_{i,j} = \omega_{j, i}$ for $i,j \ge 0$ under the measure $\widetilde \P$
\begin{align}
&\mathbb{P}\Big\{v^{(\lambda)}\Big(\Big\lfloor \frac{N}{p}\big(p + (1-p)u\big)^2\Big\rfloor\Big)> \lfloor N-N^{2/3}-1\rfloor\Big\}\notag\\
&\hspace{2cm}\leq  \widetilde{\mathbb{P}}\Big\{v^{\big(\frac{p(1-\lambda)}{ \lambda+ p(1-\lambda) }\big)}(\lfloor N-N^{2/3}-1\rfloor)\le \Big\lfloor \frac{N}{p}\big(p + (1-p)u\big)^2\Big\rfloor\Big\}.\label{eq:fc0}
\end{align}
Under measure $\widetilde\P$  the environment is still i.i.d.\ and the only change is the alternation of parameter values on the boundaries. Moreover, in the transposed environment, the new competition interface $\varphi^{(\frac{p(1-\lambda)}{ \lambda+ p(1-\lambda) })}$ constructed using \eqref{compint}, would be above (as a curve) from the transposed competition interface $\varphi^{(\lambda)}$, so it would still exit from the north boundary. (see \ref{fig:whole3}).
From \eqref{lamu} substitute $u$ as a function  of $\lambda$,
\begin{figure}%
\begin{subfigure}[t]{0.5\linewidth}
\centering
\begin{tikzpicture}[>=latex, scale=0.4]
\draw [fill=sussexg!20,sussexg!20] (0,0)--(0,1)--(1,1)--(1,2)--(2,3)--(3,3)--(3,4)--(4,5)--(6,5)--(7,6)--(7,7)--(9,7)--(10,8)--(10,9)--(0,9)--(0,0);
\draw[->] (-0.2,0) -- (10,0) ;
  \draw[->] (0,-0.2) -- (0,10) ;
\draw (9,0)--(9,9)--(0,9);

\draw[blue, line width=2pt] (0,0)--(0,1)--(1,1)--(1,2)--(2,3)--(3,3)--(3,4)--(4,5)--(6,5)--(7,6)--(7,7)--(9,7)--(9,9)node[xshift=-2.5em,yshift=0.5em,black]{\tiny$(m_u(N),n_u(N))$};
\draw[blue,dashed] (9,7)--(10,8)--(10,9)node[xshift=3.3em,black]{\tiny$(v^{(\lambda)}(n_u(N)),n_u(N))$};
\draw (5,3)node[xshift=0.5em,yshift=3em,black]{$\varphi^{(\lambda)}$};
\end{tikzpicture}
\end{subfigure}%
\begin{subfigure}[t]{0.5\linewidth}
\centering
\begin{tikzpicture}[>=latex, scale=0.4]
\draw [fill=sussexp!20,sussexp!20](0,0)--(0,2)--(1,3)--(2,3)--(3,4)--(3,5)--(4,6)--(4,7)--(5,7)--(5,8)--(6,8)--(6,9)--(7,9)--(8,10)--(9,10)--(9,0)--(0,0);
\draw[->] (-0.2,0) -- (10,0) ;
  \draw[->] (0,-0.2) -- (0,10) ;
\draw (9,0)--(9,9)--(0,9);

\draw[blue, line width=2pt] (0,0)--(1,0)--(1,1)--(2,1)--(3,2)--(3,3)--(4,3)--(5,4)--(5,6)--(6,7)--(7,7)--(7,9)--(9,9)node[xshift=2.5em,black]{\tiny$(\tilde m_u(N),\tilde n_u(N))$};
\draw[sussexg,line width=2 pt] (0,0)--(0,2)--(1,3)--(2,3)--(3,4)--(3,5)--(4,6)--(4,7)--(5,7)--(5,8)--(6,8)--(6,9)--(7,9)--(8,10)--(9,10)node[xshift=5em,black]{\tiny$(v^{(\frac{p(1-\lambda)}{ \lambda+ p(1-\lambda) })}(\tilde n_u(N)),\tilde n_u(N))$};
\draw[nicos-red, line width=2pt] (0,0)--(0,3)--(1,4)--(2,4)--(2,5)--(3,6)--(3,7)--(4,8)--(4,9)--(9,9);
\draw (3,5)node[xshift=2.5em,yshift=-3em,black]{\tiny$\varphi^{(\frac{p(1-\lambda)}{ \lambda+ p(1-\lambda) })}$};
\draw (3,8)node{$\tilde \pi^{*}$};
\draw[->,sussexg](3,5)--(9.7,5)node[xshift=2.5em,yshift=0.5em,black]{\tiny$\tilde\varphi^{(\frac{p(1-\lambda)}{ \lambda+ p(1-\lambda) })}$};
\end{tikzpicture}
\end{subfigure}
\caption{Comparison of various curves in $ \omega_{i,j}$ and  $\tilde \om_{i,j} =\omega_{j, i}$ environments. The thickset blue curve (color online) in the left figure is the competition interface in $\om_{i,j}$ and the reflected curve can be seen in the same color to the right. The green curve is the competition interface in $\tilde \om_{i,j}$ weights which is higher than the reflected $\varphi$ and the red curve is the right-most maximal paths in the reversed $\tilde \om_{i,j}^*$ weights with boundaries on north and east, which is higher than both the other curves.  \label{fig:whole3}}
\end{figure}
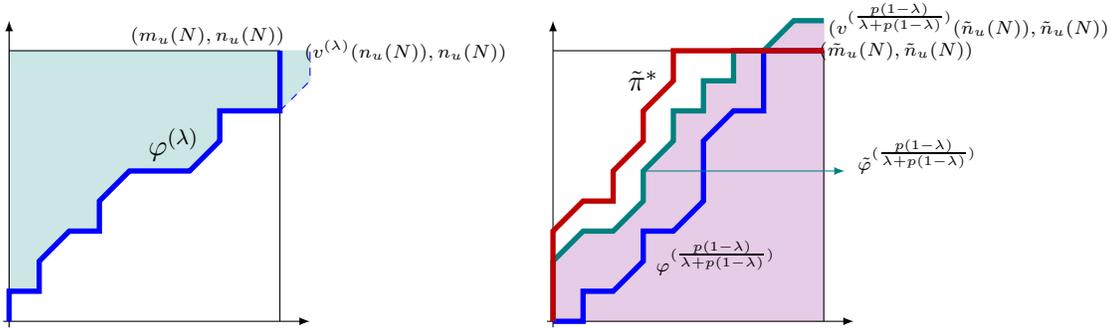

\begin{align*}
\text{Proability in \eqref{eq:fc0}}=& \widetilde{\mathbb{P}}\Big\{v^{\big(\frac{p(1-\lambda)}{ \lambda+ p(1-\lambda) }\big)}(\lfloor N-N^{2/3}-1\rfloor)\\
&\le \Big\lfloor \frac{N}{p}\big(p + (1-p)\lambda\big)^2-\frac{2}{p}(p + (1-p)\lambda)(1-p)rN^{2/3}+o(N^{2/3})\Big\rfloor\Big\}.
\end{align*}
Define $N'$ as  
\[
N'= N-N^{2/3}-1\implies N=N'+N'^{2/3}+o(N'^{2/3}).
\]
Replace $N$ with $N'$ in the probability above to obtain 
\[
\text{Probability in }\eqref{eq:fc0} \leq  \widetilde \P \Big\{v^{\big(\frac{p(1-\lambda)}{ \lambda+ p(1-\lambda) }\big)}(\lfloor N'\rfloor)\le \Big\lfloor\frac{N'}{p}(p + (1-p)\lambda)^2\Big\rfloor-KN'^{2/3}\Big\},
\]
where $K=p^{-1}(p+(1-p)\lambda)(2(1-p)r-(p+(1-p)\lambda))$ which is positive for $r$ large enough.  
Using \eqref{xire} and  \eqref{eqdis}
\begin{align*}
\text{Probability}\eqref{eq:fc0}&=\widetilde \P\big\{\xi_{e_1}^{*\big(\frac{p(1-\lambda)}{ \lambda+ p(1-\lambda) }\big)}(N')>KN'^{2/3}]=\widetilde \P \big\{\xi_{e_1}^{\big(\frac{p(1-\lambda)}{ \lambda+ p(1-\lambda) }\big)}(N')>KN'^{2/3}]\\
&\leq CK^{-3},
\end{align*}
where the last inequality follows from Corollary \ref{cor:CC}.We are now ready to prove the lemma. Start with a fixed $\eta > 0$. Then, fix an $r$ large enough so that  $CK^{-3}<\eta$ and probability \eqref{2fte} is controlled. This also imposes a restriction on the smallest value of $N$ that we can take, since we must have $\lambda < 1$. Under a fixed $r$, we can modulate $\delta,b$ and select them small enough, so that \eqref{eq:p5} holds. Finally, make $\delta$ smaller so that $C\delta^3(b^{-3}+b^{-6})<\eta$ and probability \eqref{ste} is also controlled. Thus, unifying all these results we have
\be\label{eq:Zpiures}
\mathbb{P}\{0\leq\xi^{(u)}_{e_1}(N)\leq\delta N^{2/3}\}\leq2\eta.
\ee
Note that by shrinking $\delta$ while $b$ remains fixed, \eqref{eq:p5} is reinforced.
%
This concludes the proof of the lemma.
\end{proof}

%
%
%
%
%
%
%

\begin{proof}[Proof of Theorem \ref{thm:varub}, lower bound] 
We first claim that
\be\label{eq:b-om}
\mathcal{A}_{\cN^{(u)}} = \E\Big( \xi - \sum_{i = 1}^{\xi}\om_{i,0} \Big) = \E\Big(\sum_{i = 1}^{\xi}(1 - \om_{i,0} )\Big).
\ee
Under this claim,  we can write 
\begin{align*}
\mathcal{A}_{\cN^{(u)}} &= \E\Big(\sum_{i = 1}^{\xi_{e_1}^{(u)}}(1 - \om_{i,0} )\Big)\\
& \ge  \E\Big(\mathbbm{1}\{ \xi_{e_1}^{(u)}(N) \ge \delta N^{2/3} \}\sum_{i = 1}^{\fl{\delta N^{2/3}}}(1 - \om_{i,0} )\Big)\\
&\ge \alpha N^{2/3} \P\Big\{ \xi_{e_1}^{(u)}(N) \ge \delta N^{2/3},\,\, \sum_{i = 1}^{\fl{\delta N^{2/3}}}(1 - \om_{i,0} ) \ge \alpha N^{2/3} \Big\}.
\end{align*}
Fix an $\eta$ positive and smaller than $1/4$. Now, by making $\delta$ sufficiently small, we can make the event $\{ \xi_{e_1}^{(u)}(N) \ge \delta N^{2/3} \}$ have probability larger than $1 -\eta$ by Lemma \ref{lem:asym}, for $N$ sufficiently large. With $\delta$ fixed, we can make $\alpha$ smaller, so that the event $\Big\{\sum_{i = 1}^{\fl{\delta N^{2/3}}}(1 - \om_{i,0} ) \ge \alpha N^{2/3} \Big\}$ also has probability larger than $1 -\eta$. Therefore their intersection has probability greater than  $1 - 2\eta$. 

By Proposition \ref{lem:exit} and the fact that we are in a characteristic direction, the result follows.

It now remains to verify  \eqref{eq:b-om}. Using the fact that 
\[ 
\mathscr H^{(\e)}_{i,0}\vee \om_{i,0} - \om_{i,0} = \mathscr H^{(\e)}_{i,0} -  \mathscr H^{(\e)}_{i,0} \om_{i,0}, 
\]
we write using \eqref{eq:exn-n}
\begin{align*}
\E_{\P\otimes\mu_\e}(\mathcal N^{u_\e} - \mathcal N^{(u)} ) &\ge \E_{\P\otimes\mu_\e}\Big( \mathscr S^{u_\e}_{\xi^{(u)}_{e_1}} -  \mathscr S^{(u)}_{\xi^{(u)}_{e_1}}\Big) = \E_{\P\otimes\mu_\e}\Big( \sum_{i=1}^{\xi^{(u)}_{e_1}} \mathscr H^{(\e)}_{i,0} -  \mathscr H^{(\e)}_{i,0} \om_{i,0}\Big)\\
&=\e\E(\xi^{(u)}_{e_1}) -\E_{\P\otimes\mu_\e}\Big( \sum_{i=1}^{\xi^{(u)}_{e_1}}\mathscr H^{(\e)}_{i,0} \om_{i,0}\Big)\\
&= \e\E(\xi^{(u)}_{e_1}) -\E_{\P\otimes\mu_\e}\Big(\sum_{y = 1}^{m_{u}(N)} \sum_{i=1}^{y}\mathscr H^{(\e)}_{i,0} \om_{i,0}\mathbbm1\{ \xi^{(u)}_{e_1}=y\}\Big)\\
&= \e\E(\xi^{(u)}_{e_1}) -\E_{\P\otimes\mu_\e}\Big(\sum_{i = 1}^{m_{u}(N)} \mathscr H^{(\e)}_{i,0} \om_{i,0}\mathbbm1\{ \xi^{(u)}_{e_1} \ge i \}\Big)\\
&= \e\E(\xi^{(u)}_{e_1}) - \e \E\Big(\sum_{i = 1}^{m_{u}(N)} \om_{i,0}\mathbbm1\{ \xi^{(u)}_{e_1} \ge i \}\Big) =  \e\E(\xi^{(u)}_{e_1}) - \e \E\Big(\sum_{i = 1}^{ \xi^{(u)}_{e_1} } \om_{i,0}\Big).
\end{align*}
Combine the expectations and divide by $\e$. Then take a limit as $\e \to 0$ to finish the proof.
\end{proof}

\section{Variance in off-characteristic directions}
\label{sec:offchar}
In this section we want to deduce the central limit theorem for rectangles that do not have characteristic shape. 

\begin{proof}[Proof of Theorem \ref{thrm:voffchar}] We prove the theorem in the case $c<0$, analogue arguments follow for $c>0$. Set $m^*_u(N)=m_u(N)+\fl{c N^\alpha}$. Now, the point $(m^*_u(N), n_u(N)+\fl{c N^\alpha})$ is in the characteristic direction. Thus
\[G^{(u)}_{m_u(N), n_u(N)+\fl{c N^\alpha}}=G^{(u)}_{m^*_u(N), n_u(N)+\fl{c N^\alpha}}+\sum_{i=m^*_u(N)+1}^{m_u(N)}I_{i, n_u(N)+\fl{c N^\alpha}}.\]
Note that the second the term on the right hand side is a sum of $m_u(N)-m_u^*(N)=\fl{c N^\alpha}$  i.i.d Bernoulli distributed with parameter $\lambda$. We center by the mean of each random variable and we indicate them with a bar over the random variable. Multiply both sides by $N^{-\alpha/2}$ to obtain
\[N^{-\alpha/2}\bar G^{(u)}_{m_u(N), n_u(N)+\fl{c N^\alpha}}=N^{-\alpha/2}(\bar G^{(u)}_{m^*_u(N), n_u(N)+\fl{c N^\alpha}}+\sum_{i=m^*_u(N)+1}^{m_u(N)}\bar I_{i, n_u(N)+\fl{c N^\alpha}}).\]
The first term on the right hand side is stochastically $O(N^{1/3-\alpha/2})$. Since $\alpha>2/3$ this term converges to zero in probability. On the other hand the second term satisfies a CLT.
\end{proof}
Note that for any $\lambda \in (0,1)$, for any $\e>0$ the endpoint $(N,pN-\e N)$ (resp. $(N,N/p + \e N)$) will always be the north-east corner of an off-characteristic rectangle - no matter what the value of $\lambda$.  

\section{Variance without boundary} 
\label{sec:noboundary}
In this section we prove some results for the last passage time in the model without boundaries but still with fixed endpoint. We begin reminding the last passage time of the model without boundaries to reach a point in the characteristic direction \eqref{eq:bound} is $G_{(1,1),(m_u(N),n_u(N))}$ and the last passage time of the model with boundaries to reach the same point is $G^{(u)}_{(0,0),(m_\lambda(N),n_\lambda(N))}$. We want to prove another version of Lemma \ref{lem:probban}.
\begin{lemma}\label{lem:bulk}
Fix $0<\alpha<1$. Then there exist a positive integer $N_0=N(b,u)$ and constant $C=C(\alpha,u)$ such that, for all $N\geq N_0$ and $b\geq C_0$ we have
\[\P\{G^{(u)}_{(0,0),(m_\lambda(N),n_\lambda(N))}-G_{(1,1),(m_u(N),n_u(N))}\geq bN^{1/3}\}\leq Cb^{-3\alpha/2}.\]
\end{lemma}
\begin{proof}
We prove only the case where the maximal path exits from the $x$-axis. Similar arguments hold for the maximal path exits from the $y$-axis and find the same bound.

Note that
\begin{align}
&\P\{G^{(u)}_{(0,0),(m_\lambda(N),n_\lambda(N))}-G_{(1,1),(m_u(N),n_u(N))}\geq bN^{1/3}\}\notag\\
&\hspace{1cm}\leq\P\Big\{\sup_{1\leq z\leq aN^{2/3}}\big\{\sS^{(u)}_z+G_{(z,1),(m_u(N),n_u(N))}-G_{(1,1),(m_u(N),n_u(N))}\big\}\geq bN^{1/3}\Big\}\label{probnb1}\\
&\hspace{3cm}+\P\Big\{\sup_{1\leq z\leq aN^{2/3}}\big\{\sS^{(u)}_z+G_{(z,1),(m_u(N),n_u(N))}\big\}\not= G_{(1,1),(m_u(N),n_u(N))}\Big\}.\label{probnb2}\
\end{align}
For \eqref{probnb2} using \ref{cor:CC}, there exists a $C=C(u)$ such that
\be\label{ppb2}
\begin{aligned}
\P\Big\{\sup_{1\leq z\leq aN^{2/3}}\big\{\sS^{(u)}_z+G_{(z,1),(m_u(N),n_u(N))}\big\}&\not= G_{(1,1),(m_u(N),n_u(N))}\Big\}\\
&\leq\P[\xi_{e_1}^{(u)}(N)\geq aN^{2/3}]\leq Ca^{-3}.
\end{aligned}
\ee

For\eqref{probnb1} we use the results from the proof of Lemma \ref{lem:probban}. Define
\[\lambda=u-rN^{-1/3}\]
From \eqref{plb1} and \eqref{eq:Doob}, where we choose $a=b^{\alpha/2},d=2,$ and $r=b^{\alpha/2}$ we have the upper bound 
\begin{equation}\label{eq:Doobmod}
\mathbb{P}\Big\{\sup_{1\leq z\leq aN^{2/3}}\{\sS^{(u)}_{z-1}-\mathcal V^{(\lambda)}_{z-1}\}\geq bN^{1/3}- 1\Big\}\leq\frac{C(\alpha, u)b^{\alpha/2}}{(b-b^\alpha-N^{-1/3})^2}
\end{equation}
where $C(\alpha,u)>0$ is large enough so that for $b\geq C$ \eqref{eq:condonb} is satisfied and the denominator in \eqref{eq:Doobmod} is at least $b/2$. Then we can claim that for all $b\geq C$ and $N\geq N_0=4^3b^{-3}$

\begin{align*}
\P\Big\{\sup_{1\leq z\leq aN^{2/3}}\big\{\sS^{(u)}_z&+G_{(z,1),(m_u(N),n_u(N))}-G_{(1,1),(m_u(N),n_u(N))}\big\}\geq bN^{1/3}\Big\}\\
&\hspace{1cm}\leq\mathbb{P}\Big\{v^{(\lambda)}\Big(\Big\lfloor \frac{N}{p}\big(p + (1-p)u\big)^2\Big\rfloor\Big)\le \lfloor N\rfloor\Big\}+Cb^{\alpha/2-2}.
\end{align*}
Since $N\geq N_0$ we can use the result \eqref{eq:vresult} and remembering that $r=b^{\alpha/2}$ in this case we obtain
\be\label{ppb1}
\begin{aligned}
\P\Big\{\sup_{1\leq z\leq aN^{2/3}}\big\{\sS^{(u)}_z+G_{(z,1),(m_u(N),n_u(N))}&-G_{(1,1),(m_u(N),n_u(N))}\big\}\geq bN^{1/3}\Big\}\\
&\hspace{2.5cm}\leq Cb^{-3\alpha/2}+Cb^{\alpha/2-2}.
\end{aligned}
\ee
Combining \eqref{ppb1} and \eqref{ppb2} we obtain the final result.
\end{proof}

All the constants which will be defined in this section depend on the values $x,y$ and $p$.

\begin{proof}[Proof of Theorem \ref{thm:nbsl}]
By Chebyshev, Theorem \ref{thm:varub} for the lower bound, Lemma \ref{lem:bulk}
\begin{align*}
\P\{|&G_{(1,1),(\lfloor Nx\rfloor,\lfloor Ny\rfloor)}-Ng_{pp}(x,y)|\geq bN^{1/3}\}\\
&\leq \P\{|G_{(1,1),(m_u(N),n_u(N))}-G^{(u)}_{(0,0),(m_\lambda(N),n_\lambda(N))}|\geq \frac{1}{2}bN^{1/3}\}\\
&\hspace{6cm}+\P[|G^{(u)}_{(0,0),(m_\lambda(N),n_\lambda(N))}-Ng_{pp}(x,y)|\geq\frac{1}{4}bN^{1/3}]\\
&\leq Cb^{-3\alpha/2}+Cb^{-2}\leq Cb^{-3\alpha/2}.
\end{align*}
To get the moment bound, 
\[
\E\bigg[\bigg|\frac{G_{(1,1),(\lfloor Nx\rfloor,\lfloor Ny\rfloor)}-Ng_{pp}(x,y)}{N^{1/3}}\bigg|^r\bigg]=\int_0^\infty\P\bigg[\bigg|\frac{G_{(1,1),(\lfloor Nx\rfloor,\lfloor Ny\rfloor)}-Ng_{pp}(x,y)}{N^{1/3}}\bigg|^r\geq b^r\bigg]db.
\]
At this point using \eqref{eq:ppp} where $b$ in this case is $b^{1/r}$ 
\[\int_0^\infty\P\bigg[\bigg|\frac{G_{(1,1),(\lfloor Nx\rfloor,\lfloor Ny\rfloor)}-Ng_{pp}(x,y)}{N^{1/3}}\bigg|^r\geq b^r\bigg]db\leq C_0 + \int_{C_0}^\infty Cb^{\frac{-3\alpha}{2r}}db<\infty,\]
which converges iff $1\leq r <3\alpha/2$.
\end{proof}

\subsection{Variance in flat-edge directions without boundary} 
 
 We only treat explicitly the case for which $y\leq px$. Since our model is symmetric, the same arguments can be repeated to prove the case $y\geq \frac{1}{p}x$.

%

We force macroscopic distance from the critical line, i.e. we assume that we can find $\e > 0$ so that the sequence of endpoints $(N, n(N))$ satisfy
\be \label{eq:fatbelow}
\varlimsup_{n \to \infty} \frac{ n(N) }{ N } \le p - \e. 
\ee
%

\begin{proof}[Proof of Theorem \ref{thm:flatvar}]
Consider the following naive strategy:  
We construct an approximate maximal path $\pi$ for  $G_{N, n(N)}$, knowing that for large $n(N)< \fl{(p - \e/2)N}$ without using the boundaries. $\pi$ enters immediately inside the bulk and moves right until it finds a weight  to collect diagonally. After that this procedure repeats. For each iteration of this procedure, the horizontal length of this path increases by a random Geometric($1/p$) length, independently of the past. 

The probability  that $\pi$ will take more than $N$ steps before reaching level $n(N)$ is the same as the probability that the sum of $n(N)$ independent $X_i \sim$ Geometric$(1/p)$ r.v.'s exceeds the value $N$ which is a large deviation event. In symbols 
\[
\P\{ G_{N, n(N)}(\pi) < n(N)\} = \P\Big\{ \sum_{i=1}^{n(N)}X_i > N \Big\} \le \P\Big\{ \sum^{\fl{(p - \e/2)N}}_{i=1}X_i > N \Big\} \le e^{-cN}.
\] 
Now, let $\mathcal A = \{  G_{N, n(N)}(\pi) = n(N) \}$. 
\begin{align*}
\text{Var}(G_{N, n(N)}) &= \E(G_{N, n(N)}^2) - (\E(G_{N, n(N)}))^2 \\
&\le (n(N))^2 - (\E(G_{N, n(N)}\mathbbm 1_{\mathcal A})^2 = (n(N))^2 - (n(N))^2 \P\{ \mathcal A\}^2\\
&\le (n(N))^2( 1 - (1 - e^{-cN})^2) \le C N^2 e^{-cN} \to 0. \qedhere
\end{align*}
\end{proof}

\section{Fluctuations of the maximal path in the boundary model}
 \label{sec:paths}
In this last section we prove the path fluctuations in the characteristic direction in the model with boundaries. The idea behind it is to study how long the maximal path spends on any horizontal (or vertical) level and find a bound for the distance between the maximal path and the line which links the starting and the ending point which corresponds to the macroscopic maximal path.

 
 
Fix a boundary parameter $\lambda$ and for this section the characteristic direction  $(m_{\lambda}(N),n_{\lambda}(N))$  is abbreviated by $(m,n)$ and it is the endpoint for the maximal path. Consider two rectangles $\mathcal R_{(k,\ell),(m,n)}\subset\mathcal R_{(0,0),(m,n)}$ with $0<k<m_{\lambda}(N)$ and $0<\ell<n_{\lambda}(N)$. In the smaller rectangle $\mathcal R_{(k,\ell),(m_{\lambda}(N),n_{\lambda}(N))}$ impose boundary conditions on the south and west edges given by the distributions defined in Lemma \ref{burke}.
\begin{equation}\label{eq:newb}
I_{i,\ell}\dis I_{i,0}\quad J_{k,j}\dis J_{0,j}\text{ with } i\in\{k+1,\dots,m\},j\in\{\ell+1,\dots, n\}.
\end{equation}

Recall that  \eqref{v1} and \eqref{v2} define respectively the $i$ coordinate where the maximal path enters and exits from a fixed horizontal level $j$. Since we are interested in studying either the horizontal and vertical fluctuations we also define the $j$ coordinate where the maximal path enters and exits from a fixed vertical level $i$ as
\begin{equation}
w_0(i)=\min\{j\in\{0,\dots,n\}:\exists k\text{ such that } \pi_k=(i,j)\},
\end{equation}
and 
\begin{equation}\label{eq:w1}
w_1(i)=\max\{j\in\{0,\dots,n\}:\exists k\text{ such that } \pi_k=(i,j)\}.
\end{equation}

To make our notation clearer we distinguish the exit point for the path which starts from $(0,0)$ to the one which starts from $(k,\ell)$ adding the superscript $(0,0)$ or $(k,\ell)$. We define the exit point from the south edge of the rectangle $\mathcal R_{(k,\ell),(m,n)}$ as
\begin{equation}
\xi_{e_1}^{(k,\ell)}=\max_{\pi\in\Pi_{(k,\ell),(m,n)}}\{r\geq 0 : (k+i,\ell) \in \pi \text{ for }0\leq i\leq r, \pi \text{ is the right-most maximal}\}.
\end{equation}
Observe from \eqref{eq:newb} that $\xi_{e_1}^{(k,\ell)}$ and $v_1(\ell)-k$ have the same distribution, i.e.
\begin{equation}\label{eq:neweq}
\P\{\xi_{e_1}^{(k,\ell)}=r\}=\P\{v_1(\ell)=k+r\}.
\end{equation}
\begin{proof}[Proof of Theorem \ref{thm:pathflucts}]
Note that if $\tau=0$ \eqref{eq:path1} and \eqref{eq:path2} are already contained in \eqref{eq:Z} and \eqref{eq:Zpiures}.

For $0<\tau<1$ set $v=\lfloor bN^{2/3}\rfloor$ and $(k,\ell)=(\lfloor \tau m\rfloor,\lfloor \tau n\rfloor)$. We add a superscript $\P^{(\cdot,\cdot)}\{\cdot\}$ when we want to emphasise the target point for which we are computing the probability. Remember that the rectangle $\mathcal R_{(k,\ell),(m,n)}$ has boundary condition \eqref{eq:newb}. By Lemma \ref{burke} 
\begin{equation}\label{eq:neweq2}
\begin{aligned}
\P^{(m,n)}\{v_1(\lfloor \tau n\rfloor)\geq\lfloor \tau m\rfloor+v\}&=\P^{(m,n)}\{\xi_{e_1}^{(k,\ell)}\geq v\}, \quad &&\text{ by \eqref{eq:neweq}}\\
&=\P^{(m-k,n-\ell)}\{\xi_{e_1}^{(0,0)}\geq v\}, \quad &&\text{ by \eqref{eq:gradients},  \eqref{eq:newb}} .
\end{aligned}
\end{equation}
Note that $(m-k,n-\ell)$ is still in the characteristic direction since $(m-k,n-\ell)=(1-\tau)(m,n)$. Therefore, from \eqref{eq:neweq2} and Corollary \ref{cor:CC} 
\[\P^{(m,n)}\{v_1(\lfloor\tau n\rfloor)>\tau m+bN^{2/3}\}\leq C_2b^{-3}.\]
To prove the other part of \eqref{eq:path1} notice that
\be\label{imp}
\P^{(m,n)}\{v_0(\lfloor \tau n\rfloor)<\lfloor \tau m\rfloor-v\}\leq \P^{(m,n)}\{w_1(\lfloor \tau m\rfloor-v)\geq \lfloor \tau n\rfloor\}.
\ee
Let $k=\lfloor\tau m\rfloor-v$ and $\ell=\lfloor\tau n\rfloor-\lfloor nv/m\rfloor$. Then, up to integer-part corrections, $k/\ell=m/n$. For a constant $C_{\lambda}>0$ and $N$ sufficiently large , $\lfloor \tau n\rfloor\geq \ell+C_{\lambda}bN^{2/3}$. Note that from \eqref{eq:newb} we can write 
\begin{equation}\label{eq:neweq3}
\P^{(m,n)}\{\xi_{e_2}^{(k,\ell)}=r\}=\P^{(m,n)}\{w_1(k)=\ell+r\}.
\end{equation}
 The vertical analogue of \eqref{eq:neweq2} for $w_1$ is
\begin{align*}
\P^{(m,n)}\{w_1(\lfloor \tau m\rfloor -v)\geq \lfloor \tau n\rfloor\}&=\P^{(m,n)}\{w_1(k)\geq \ell+C_\lambda bN^{2/3}\}\\
&=\P^{(m,n)}\{\xi_{e_2}^{(k,\ell)}\geq C_{\lambda} bN^{2/3}\}\quad&&\text{ by }\eqref{eq:neweq3}\\
&= \P^{(m-k,n-\ell)}\{\xi_{e_2}^{(0,0)}\geq C_{\lambda}bN^{2/3}\}&&\text{ by \eqref{eq:gradients},  \eqref{eq:newb}}.
\end{align*}
Combine this last result with \eqref{imp} and from Corollary \ref{cor:CC} applied to $\xi_{e_2}$ \eqref{eq:path1} follows. 

Finally, we prove \eqref{eq:path2}. We want to compute
\[\P\{\exists \,k\text{ such that } |\hat\pi_k-(\tau m,\tau n)|\leq \delta N^{2/3}\}.\]
If the path $\hat\pi$ comes within $\ell_{\infty}$ distance $\delta N^{2/3}$ of $(\tau m,\tau n)$, then it necessarily enters through the south or west side of the rectangle $\mathcal R_{(k+1,\ell+1),(k+4\lfloor \delta N^{2/3}\rfloor,\ell+ 4\lfloor c\delta N^{2/3}\rfloor)}$ (or via a diagonal step from the south-west corner), where the point $(k,\ell)$ $=(\lfloor\tau m\rfloor-2\lfloor \delta N^{2/3}\rfloor,\lfloor\tau n\rfloor-2\lfloor c\delta N^{2/3}\rfloor$ and the constant $c>m/n$ for large enough $N$. The constant $c$ is there to make the rectangle of characteristic shape.

From the perspective of the rectangle $\mathcal R_{(k,\ell), (m,n)}$ this event is equivalent to either $0\le \xi_{e_1}^{(k,\ell)}\leq 4\delta N^{2/3}$ or $0\le \xi_{e_2}^{(k,\ell)}\leq 4c\delta N^{2/3}$. For these reasons we have
\begin{align*}
\P^{(m,n)}\{\exists k &\text{ such that } |\hat\pi_k-(\tau m,\tau n)|\leq \delta N^{2/3}\}\\
&\leq \P^{(m,n)}\{0<\xi_{e_1}^{(k,\ell)}\leq 4\delta N^{2/3}\text{ or }0<\xi_{e_2}^{(k,\ell)}\leq 4c\delta N^{2/3}\}\\
&= \P^{(m-k,n-l)}\{0\le \xi^{(0,0)}_{e_1}\leq 4\delta N^{2/3}\text{ or }0 \le \xi^{(0,0)}_{e_2}\leq 4c\delta N^{2/3}\}.
\end{align*}
We get the result using equation \eqref{eq:Zpiures} for both exit points.
\end{proof}
  
\bibliographystyle{plain}

\end{document}